\documentclass[graybox]{svmult}

\usepackage{amssymb}
\usepackage{amsmath}

\usepackage{amsfonts}
\usepackage{latexsym} 
\usepackage{ wasysym}
\usepackage{newlfont}
\usepackage{titlesec}
\titlespacing*{\section}{0pt}{1.1\baselineskip}{\baselineskip}
\titlespacing*{\subsection}{0pt}{1.1\baselineskip}{\baselineskip}

\usepackage{stmaryrd}
\usepackage{mathptmx}       
\usepackage{type1cm}        
\usepackage{makeidx}         
\usepackage{graphicx}        
\usepackage{multicol}        
\usepackage[bottom]{footmisc}

\usepackage{xspace}
\newcommand{\gismo}{{\fontfamily{phv}\fontshape{sc}\selectfont G\pmb{+}Smo}\xspace}

\makeindex             

\begin{document}

\title*{Multipatch Discontinuous Galerkin Isogeometric Analysis}
\author{Ulrich Langer,  Angelos Mantzaflaris, Stephen E. Moore  
and Ioannis Toulopoulos}
\institute{Ulrich Langer, Angelos Mantzaflaris, Stephen E. Moore,  
Ioannis Toulopoulos \at RICAM, Altenbergerstr. 69, A-4040 Linz, Austria, 
\email{(ulrich.langer, angelos.mantzaflaris, stephen.moore, ioannis.toulopoulos)@ricam.oeaw.ac.at}
}

\maketitle

\abstract{Isogeometric analysis (IgA) uses the same class of basis functions for both, 
representing the geometry of the computational domain and approximating the 
solution. In practical applications, geometrical patches are used in order to 
get flexibility in the geometrical representation. 
This multi-patch representation corresponds to 
a decomposition of the computational domain into non-overlapping subdomains 
also called patches in the geometrical framework.
We will present  discontinuous Galerkin (dG) methods that allow for 
discontinuities across the subdomain (patch) boundaries. The required 
interface conditions are weakly imposed 
by the dG terms associated with the boundary of the 
sub-domains. The construction and the corresponding discretization error analysis of 
such dG multi-patch IgA schemes will be given for heterogeneous diffusion model problems 
in volumetric 2d and 3d domains as well as on open and closed surfaces.
The theoretical results are confirmed by numerous 
numerical experiments which have been performed in \gismo.
The concept and the main features of the IgA library  \gismo
are also described.
}

\section{Indroduction}
\label{LMMT_sec:1:Indroduction}
%
The Isogeometric analysis (IgA), which was introduced  by Hughes, Cottrell and Bazilevs in 2005, 
is a new discretization 
technology
which uses the same class of basis functions for both, 
representing the geometry of the computational domain and 
approximating the solution of problems modeled by 
Partial Differential Equations (PDEs), 
\cite{LMMT_HughesCottrellBazilevs:2005a}. 
IgA uses the exact geometry in the class of Computer Aided Design (CAD) geometries, 
and thus geometrical errors 
introduced by approximation of the physical domain are eliminated.
This feature is especially important in technical applications
where the CAD geometry description is directly used in the production process.
Usually, IgA uses basis functions like B-Splines and Non-Uniform Rational B-Splines (NURBS), 
which are standard in CAD, 
and have several advantages which make them suitable for 
efficient and accurate simulation,
see \cite{LMMT_CottrellHughesBazilevs:2009a}. 
The mathematical analysis of  approximation properties, 
stability and discretization error estimates 
of NURBS spaces have been well studied in 
\cite{LMMT_BazilevsBeiraoCottrellHughesSangalli:2006a}. 
Furthermore, 
approximation error estimates 
due to mesh, polynomial degree and smoothness refinement
have been obtained 
in \cite{LMMT_BeiraoBuffaRivasSangalli:2010a}, 
and 
likewise a hybrid method that combines a globally $C^{1}-$continuous, 
piecewise polynomial finite element basis with rational NURBS-mappings 
have also been considered in 
\cite{LMMT_KleissJuettlerZulehner:2012a}.

In practical applications, the computational domain $\Omega \subset \mathbb{R}^d$ ($d=2,3$)
is usually represented by multiple patches leading to non-matching meshes 
and thus to patch-wise non-conforming approximation spaces. 
In order to handle non-matching meshes and polynomial degrees across the 
patch interfaces, the discontinuous Galerkin (dG) technique 
that is now well established in the Finite Element Analysis (FEA) of 
different field problems is employed. 
Indeed, dG methods have been developed and analyzed for many applications including elliptic, 
parabolic and hyperbolic PDEs. 
The standard dG finite element methods use approximations 
that are discontinuous across the boundaries of every finite element of the triangulation. 
To achieve consistent, stable and accurate schemes, 
some conditions are prescribed on the inter-element boundaries, 
see, e.g., Rivi\`ere's monograph \cite{LMMT_Riviere:2008a}.

In this paper, we present and analyze multipatch dG IgA methods 
for solving elliptic PDEs in volumetric 2d or 3d computational domains 
as well as on open and closed surfaces. 
As model problems, we first consider  diffusion problems of the form
\begin{equation}
\label{LMMT_sec:1:eqn:Model1}
   -\nabla  \cdot (\alpha\nabla u) = {f}{\ }  \text{in}{\ } \Omega, {\ }
    \text{and}{\ }{u}     = {u}_D {\ }\text{on $\partial \Omega $}, 
\end{equation}
where $\Omega$ is a bounded Lipschitz domain in $\mathbb{R}^d,{\ }d=2,3$,
with the boundary $\partial \Omega$,
$f$ is a given source term, and $ u_D$ are given Dirichlet data.
Neumann, Robin and mixed boundary conditions can easily be 
treated in the same dG IgA framwork which we are going to analyze in this paper.
The same is true for including a reaction term into the PDE  (\ref{LMMT_sec:1:eqn:Model1}).
The diffusion coefficient $ \alpha $ is assumed to be bounded from above and below by 
strictly positive constants. We allow $ \alpha $ to be discontinuous 
across the inter-patch boundaries with possible large jumps across these interfaces. 
More precisely, 
the computational domain $\Omega$ is subdivided into
a  collection $\mathcal{T}_H(\Omega) := \{\Omega_i\}_{i=1}^N$
of non-overlapping subdomains (patches) $\Omega_1,\ldots,\Omega_N$,
where the patches $\Omega_i$ are obtained by some NURBS mapping $\Phi_i$ 
from the parameter domain 
$\widehat{\Omega}= (0,1)^d \subset \mathbb{R}^d$.
In every subdomain, the problem is discretized under the IgA methodology.
The dG IgA method is considered in a general case without  imposing any 
matching grid conditions and without any continuity requirements for the discrete solution 
across the interfaces $F_{ij}=\partial \Omega_i\cap \partial\Omega_j$. 
Thus, on the interfaces,  the dG technique
ensures
the communication of the solution between the adjacent subdomains.
For simplicity, we assume  that the coefficients are piecewise constant
with respect to the decomposition  $\mathcal{T}_H(\Omega_i)_{i=1}^N$.
It is well known that solutions of elliptic problems with discontinuous coefficients are
generally not smooth 
\cite{LMMT_Kellog:1975a, LMMT_Grisvard:2011a},  
Thus, we cannot expect that the dG IgA provides  
convergence rates like in the smooth case.
Here, the reduced regularity of the solution is of primary concern.
Following the techniques developed by Di Pietro and Ern
\cite{LMMT_DiPietroErn:2012a,LMMT_DiPietroErn:2012b} 
for dG FEA,
we present an error analysis for diffusion problems 
in 2d and 3d computational domains with solutions belonging to  
$W^{l,p}(\Omega_i),\, l\geq 2,{\ } p\in (\max\{1,{2d}/{(d+2(l-1))}\},2],\, i=1,\ldots,N$.        
In correspondence with this regularity, we show optimal convergence rates of the discretization 
error in the  classical ``broken'' dG - norm. 
Due to the peculiarities of the dG IgA, the proofs are quite technical and can be 
found in the paper \cite{LMMT_LangerToulopoulos:2014a}
written by two of the co-authors.
Solutions of elliptic problems with low regularity can also appear in other cases.
For example, geometric singular boundary points (re-entrant corners),
points with change of the boundary conditions, and singular source terms
can cause low-regularity solutions as well, see, e.g., \cite{LMMT_Grisvard:2011a}. 
Especially, in the IgA, we want to use the potential of 
the approximation properties of NURBS based on high-order polynomials.
In this connection, the error estimates obtained for the low-regularity case
can be very useful. Indeed, via mesh grading, that has been used in the FEA 
for a long time 
\cite{LMMT_OganesjanRuchovetz:1979a,LMMT_ApelMilde1996},
we can recover the 
full convergence rate corresponding to the underlying polynomial degree.
We mention that in the literature other techniques have been proposed 
for solving two dimensional problems with  singularities.  For example, in 
\cite{LMMT_OhKimJeong:2013a} and \cite{LMMT_JeongOhKangKim:2013a}, 
the original  B-Spline finite dimensional space has been  enriched by  generating singular  
functions which resemble the types of the singularities of the problem. 
Also in \cite{LMMT_SangalliaNurbs2012}, by studying the anisotropic character of the 
singularities of the problem, the one dimensional approximation properties of
the B-Splines are generalized in order to produce anisotropic refined meshes in the regions of the singular
points.
%
Our error analysis is accompanied by a series of 
numerical tests which fully confirm our convergence
rate estimates for the multipatch dG IgA in both the full and the low regularity 
cases as well as the recovering of the full rate by means of mesh grading techniques.
It is very clear that our analysis can easily be  generalized to more general classes 
of elliptic boundary value problems like plane strain or stress linear 
elasticity problems or linear elasticity problems in 3d volumetric 
computational domains.

We will also consider multipatch dG IgA of diffusion problems 
of the form (\ref{LMMT_sec:1:eqn:Model1})
on sufficiently smooth, open and closed surfaces $\Omega$ in $\mathbb{R}^3$,
where the gradient $\nabla$ must now be replaced by 
the surface gradient $\nabla_\Omega$,
see, e.g., \cite{LMMT_DziukElliott:2013a},
and the patches $\Omega_i$, into which $\Omega$ is decomposed,
are now images of the parameter domain 
$\widehat{\Omega}= (0,1)^2$ by the mapping
$\Phi_i: \widehat{\Omega} \subset \mathbb{R}^2 
         \rightarrow \Omega_i \subset \mathbb{R}^3$. 
The case of matching meshes was considered and analyzed 
in \cite{LMMT_LangerMoore:2014a} by two of the co-authors,
see also \cite{LMMT_ZhangXuChen2014a} for a similar work.
It is clear that the results for non-matching spaces 
and for mesh grading  presented here for the volumetric (2d and 3d) case can 
be generalized to  diffusion problems on open and closed surfaces.
Until now, the most popular numerical method 
for solving PDEs on surfaces is the surface Finite Element Method (FEM).
The surface FEM 
was first
applied to compute 
approximate solutions of the Laplace-Beltrami problem,
where the finite element solution is constructed from the 
variational formulation of the surface PDEs 
in the finite element space that is living on the triangulated surface 
approximating the real surface \cite{LMMT_Dziuk:1988a}.
This method has been extended to the parabolic equations 
fixed surfaces by Dziuk and Elliot \cite{LMMT_DziukElliott:2007b}. 
To treat conservation laws on moving surfaces, 
Dziuk and Elliot proposed the evolving surface finite element methods,
see \cite{LMMT_DziukElliott:2007a}. 
Recently, Dziuk and Elliot  have published the survey paper 
\cite{LMMT_DziukElliott:2013a}
which provides a comprehensive presentation of different finite element approaches 
to the solution of PDEs on surfaces with several applications.
The first dG finite element scheme, which extends Dziuk's approach, 
was proposed by Dedner et al. \cite{LMMT_DednerMadhavanStinner:2013a} 
and has very recently been extended to adaptive dG surface FEM
\cite{LMMT_DednerMadhavan:2014a} 
and high-order dG finite element approximations on surfaces 
 \cite{LMMT_AntoniettiDednerMadhavanStangalioStinnerVeran:2014a}.
%
%
However, since all these approaches rely on the triangulated surface, 
they have an inherent geometrical error which becomes more complicated 
when approximating problems with complicated geometries.
This drawback of the surface FEM can be overcome by the IgA technology,
at least, in the class of CAD surfaces which can  be represented exactly 
by splines or NURBS.
Dede and Quateroni have introduced the surface IgA for fixed surfaces 
which can be represented by one patch \cite{LMMT_DedeQuarteroni:2012a}.
They presented convincing numerical results for several PDE problems 
on open and closed surfaces. However, in many practical applications,
it is not possible to represent the surface $\Omega$ with one patch.
The surface multipatch dG IgA that allows us to use patch-wise different 
approximations spaces on non-matching meshes is then the natural choice.

The new paradigm of IgA brings challenges regarding the
implementation. Even though the computational domain is partitioned
into subdomains (patches) and elements (parts of the domain delimited
by images of knot-lines or knot-planes), the information that is
accessible by the data structure (parametric B-spline patches) is
quite different than that of a classical finite element mesh (soup of
triangles or simplices).  In this realm, existing finite element
software libraries cannot easily be adapted  to the isogeometric
setting.  Apart from the data structures used, another issue is the
fact that FEA codes are focused on treating nodal shape function
spaces, contrary to isogeometric function spaces.  To provide a
unified solution to the above (and many other) issues, we present 
the Geometry and Simulation Modules 
(\gismo \footnote{\gismo: http://www.gs.jku.at}) 
library, which provides a unified, object-oriented development 
framework suitable to implement advanced isogeometric techniques, 
such as 
dG
methods.

The rest of the paper is organized as follows. 
In Section~\ref {LMMT_sec:2:Volumetric}, we present and analyze 
the multipatch dG IgA for diffusion problems in volumetric 
2d and 3d computational domains including the case of low-regularity
solutions. We also study the mesh grading technology 
which allows us to recover the full convergence rates 
defined by the degree $k$ of the underlying polynomials.
The numerical analysis is accompanied by numerical experiments
fully confirming the theoretical results.
Section~\ref{LMMT_sec:3:Surfaces} is devoted to multipatch dG IgA 
of diffusion problems on open and closed surfaces.
We present and discuss  new numerical results including the case 
of jumping coefficients and non-matching meshes.
The concept and the main features of the \gismo library 
is presented in Section~\ref{LMMT_sec:5:G+SMO}.
All numerical experiments presented in this paper 
have been preformed in \gismo.


\section{Multipatch dG IgA for PDEs in 3d Computational Domains}
\label{LMMT_sec:2:Volumetric}

\subsection{Multipatch dG IgA Discretization}
\label{LMMT_subsec:2.1:MultipatchdG IgADiscretization}

The weak formulation of (\ref{LMMT_sec:1:eqn:Model1}) reads as follows: 
find  a function 
$u$ from the Sobolev space 
$W^{1,2}(\Omega)$
such that $u=u_D$ 
on $\partial \Omega$ and  satisfies the variational formulation
\begin{equation}
\label{LMMT_sec:2:eqn:Model1-VF}
 a(u,v)=l(v),{\ }  \forall  v \in 
W^{1,2}_{0}(\Omega) := \{v \in  W^{1,2}(\Omega):\, v=0\;\mbox{on}\; \partial \Omega\},
\end{equation}
where the bilinear form $a(\cdot,\cdot)$ and the linear form $l(\cdot)$ 
are defined by the relations
\begin{equation*}
a(u,v)=  \int_{\Omega}\alpha \nabla u \cdot \nabla v\,dx
\quad\text{and}\quad  
l(v)=\int_{\Omega}f v \,dx,
\end{equation*}
respectively. Beside Sobolev's Hilbert spaces $H^l(\Omega) = W^{l,2}(\Omega)$,
we later also need Sobolev's Banach spaces $W^{l,p}(\Omega)$, $p \in [1,\infty)$,
see, e.g., \cite{LMMT_AdamsFournier:2008}.
%
The existence and uniqueness of the solution $u$ 
of problem (\ref{LMMT_sec:2:eqn:Model1-VF}) can be derived by Lax-Milgram's Lemma 
\cite{LMMT_LCEvans:1998a}.  
In order to apply the IgA methodology 
to problem (\ref{LMMT_sec:1:eqn:Model1}), the domain
 $\Omega$ is subdivided into a union  of subdomains 
 $\mathcal{T}_H(\Omega):=\{\Omega_i\}_{i=1}^{N}$  such that

\begin{align}
\label{LMMT_1}
\bar{\Omega} &= \bigcup _{i=1}^N \bar{\Omega}_i,\quad \text{with}\quad
\Omega_i\cap \Omega_j = \emptyset, {\ } \text{if}{\ }j\neq i.
\end{align}
As we mentioned in the introduction, 
the subdivision of $\Omega$ is assumed to be compatible with the discontinuities of $\alpha$,
i.e. they are constant in the interior of 
$\Omega_i$, that is $\alpha|_{\Omega_i}:=\alpha^{(i)}$, and
 their discontinuities  appear only  
across the interfaces $F_{ij}=\partial \Omega_i \cap \partial \Omega_j$,
cf., e.g., \cite{LMMT_Dryja:2003a, LMMT_DiPietroErn:2012a,LMMT_DiPietroErn:2012b}. 
%
Throughout the paper,
we will use the notation $a\sim b$ meaning that there are positive constants $c$ and $C$ such that
$c a \leq b \leq C a$.

\par
As it is common in IgA, 
we assume  a parametric domain $\widehat{\Omega}$ of unit length, e.g., $\widehat{\Omega} = (0,1)^d$.
For any $\Omega_i$,  we associate   $d$ knot vectors $\Xi^{(i)}_n$, $n=1,...,d$,
 on  $\widehat{\Omega}$,
which 
 create a mesh $T^{(i)}_{h_i,\widehat{\Omega}} =\{\hat{E}_{m}\}_{m=1}^{M_i}$, where $\hat E_{m}$ are the micro-elements,
 see details in \cite{LMMT_HughesCottrellBazilevs:2005a}.
 We refer to $T^{(i)}_{h_i,\widehat{\Omega}}$ as the \textit{parametric  mesh of $\Omega_i$}. 
 For every $\hat{E}_m\in T^{(i)}_{h_i,\widehat{\Omega}}$,  we denote by $h_{\hat{E}_m}$ \textit{its diameter } 
 and by $h_i=\max\{h_{\hat{E}_m}\}$ \textit{the  mesh size}  of $T^{(i)}_{h_i,\widehat{\Omega}}$. 
We assume the following properties for every $T^{(i)}_{h_i,\widehat{\Omega}}$:
 \begin{itemize}
  \item quasi-uniformity: for every  $\hat{E}_m\in T^{(i)}_{h_i,\widehat{\Omega}}$ holds  ${h}_i\sim h_{\hat{E}_m}$,
   \item  for the  micro-element edges $e_{\hat{E}_m} \subset \partial \hat{E}_m$ holds  $h_{\hat{E}_m} \sim e_{\hat{E}_m}$.
  \end{itemize}
  
On every $T^{(i)}_{h_i,\widehat{\Omega}}$, we construct the finite dimensional 
space $\hat{\mathbb{B}}^{(i)}_{h_i}$ spanned by  
 ${B}$-spline basis functions of degree $k$, see \cite{LMMT_HughesCottrellBazilevs:2005a},

\begin{align}\label{LMMT_3.2aa}
 \hat{\mathbb{B}}^{(i)}_{h_i}=
span\{\hat{B}_j^{(i)}(\hat{x})\}_{j=0}^{dim(\hat{\mathbb{B}}_{h_{i}}^{(i)})}, 
\intertext{where every  basis function $\hat B_j^{(i)}(\hat{x})$ in (\ref{LMMT_3.2aa}) 
      is derived by means of tensor
           products of  one-dimensional 
$B$-spline 
basis functions, e.g.,}
 \label{LMMT_3.2bb}
\hat B_j^{(i)}(\hat{x}) = \hat B_{j_1}^{(i)}(\hat{x}_1)\cdot\cdot\cdot \hat B_{j_d}^{(i)}(\hat{x}_d).
\end{align}
For simplicity, we assume that  the basis functions
of every $\hat{\mathbb{B}}_{h_{i}}^{(i)}, i=1,...,N$, are of the same degree $k$.

 \par
 Every subdomain  $\Omega_i\in \mathcal{T}_H(\Omega)$, $i=1,...,N$, is exactly represented through a parametrization 
 (one-to-one mapping), cf. \cite{LMMT_HughesCottrellBazilevs:2005a}, having the form
\begin{align}
\label{LMMT_2_0}
 \Phi_i: \widehat{\Omega} \rightarrow \Omega_i, &\quad
 \Phi_i(\hat{x}) = \sum_{j}C^{(i)}_j \hat{B}_j^{(i)}(\hat{x}):=x\in \Omega_i,
 \end{align}
where $C_j^{(i)}$ are the control points and $\hat{x} = {\Psi}_i(x):=\Phi^{-1}_i(x)$. 
 Using  $\Phi_i$, 
we  construct a mesh $T^{(i)}_{h_i,\Omega_i} =\{E_{m}\}_{m=1}^{M_i}$ for every  $\Omega_i$,
whose vertices are the images of the vertices  of
the corresponding mesh $T^{(i)}_{h_i,\widehat{\Omega}}$ through  $\Phi_i$. 
 If $h_{\Omega_i}=\max\{h_{E_m}: {\ }E_m\in T^{(i)}_{h_i,\Omega_i}\}$ is the 
 {subdomain $\Omega_i$ mesh size}, then,  based on  definition (\ref{LMMT_2_0}) 
 of $\Phi_i$, we have the equivalence relation
 $h_i \sim  h_{\Omega_i}.$
\par
The mesh of $\Omega$ is considered to be $T_h(\Omega)=\bigcup_{i=1}^N T^{(i)}_{h_i,\Omega_i}$,
where we note that there are no matching mesh requirements on the interior interfaces 
$F_{ij}=\partial \Omega_i \cap \partial \Omega_j, i\neq j$.
For the sake of brevity in our notations,  the interior  faces of the boundary of the  
subdomains are denoted by $\mathcal{F}_{I}$ and 
 the collection of the  faces  that belong to  
           $\partial \Omega$  by $\mathcal{F}_B$, 
           e.g. $F\in \mathcal{F}_B$
           if there is a $\Omega_i$ such that $F=\partial \Omega_i \cap \partial \Omega$.
           We denote the set of all subdomain faces by $\mathcal{F}.$ 
 \par      
 Lastly, 
 we  define  the 
$B$-spline space 
$\mathbb{B}_h(\mathcal{T}_H(\Omega))=\mathbb{B}^{(1)}_{h_{1}}\times ... \times \mathbb{B}^{(N)}_{h_{N}}$ 
on $\Omega$,
where every  $\mathbb{B}^{(i)}_{h_{i}}$ is defined 
on $T^{(i)}_{h_i,\Omega_i}$ as follows
\begin{align}\label{LMMT_3}
 \mathbb{B}^{(i)}_{h_{i}}:=\{B_{h_{i}}^{(i)}|_{\Omega_i}: B_h^{(i)}(\hat{x})=
  \hat{B}_h^{(i)}\circ \Psi_i({x}),{\ }
 \forall \hat{B}_h^{(i)}\in \hat{\mathbb{B}}^{(i)}_{h_i} \}. 
\end{align}
We assume that the mappings $\Phi_{i}$ are regular in the sense 
that there exist positive constants $c_m$ and $c_M$ such that
\begin{align}\label{LMMT_2_a}
 c_m \leq |det(J_{i}(\hat{x}))| \leq c_M, {\ }\text{for} {\ }i=1,...,N, 
  {\ }\text{for all} {\ }\hat{x} \in \hat{\Omega},
\end{align}
where 
$J_{i}(\hat{x})$  denotes the Jacobian 
$\partial \Phi_{i}(x) / \partial(\hat{x})$ of the mapping $\Phi_{i}$.
%
\par
Now, for any $\hat{u}\in W^{m,p}(\hat{\Omega}), m\geq 0, p>1$,  we define the function  
\begin{align}\label{LMMT_2_b}
 \mathcal{U}(x)=\hat{u}(\Psi_i(x)),{\ }x\in \Omega_i,
\end{align}
and the following    relation holds true, see \cite{LMMT_LangerToulopoulos:2014a},
\begin{align}\label{LMMT_2_c}
 C_m  \|\hat{u}\|_{W^{m,p}(\hat{\Omega}}) \leq  \|\mathcal{U}\|_{W^{m,p}(\Omega_i)} \leq C_M \|\hat{u}\|_{W^{m,p}(\hat{\Omega})},
\end{align}
where  the constants $C_m$ and $C_M$ depending on 
$
C_m:=C_m(\max_{m_0\leq m}(\|D^{m_0}\Phi_i\|_{\infty}),$ \\$ \|det(\Psi^{'}_{i})\|_{\infty})$ and 
$C_M:=C_M(\max_{m_0\leq m}(\|D^{m_0}\Psi_i\|_{\infty}),\|det(\Phi^{'}_{i})\|_{\infty}).
$
The usefulness of inequalities  (\ref{LMMT_2_c}) in the analysis is the following: every required relation
can be proved  in the parametric domain and then using  (\ref{LMMT_2_c}) we can directly have the 
expression on the physical subdomain. 

We use the $B$-spline spaces $\mathbb{B}^{(i)}_h$ defined in (\ref{LMMT_3}) 
for approximating  the solution of (\ref{LMMT_sec:2:eqn:Model1-VF})
in every subdomain $\Omega_i$. 
Continuity requirements for $\mathbb{B}_h(\mathcal{T}_H(\Omega))$  are not imposed on the interfaces 
$F_{ij}$ of the subdomains, clearly $\mathbb{B}_h(\mathcal{T}_H(\Omega))\subset L^2(\Omega)$ but 
$\mathbb{B}_h(\mathcal{T}_H(\Omega))\nsubseteq W^{1,2}(\Omega)$.
Thus, problem (\ref{LMMT_sec:2:eqn:Model1-VF}) is discretized by discontinuous Galerkin techniques on
$F_{ij}$, see, e.g., \cite{LMMT_Dryja:2003a}.
Using the notation $v_h^{(i)}:=v_h|_{\Omega_i}$,
we define the average and the jump of $v_h$ on $F_{ij}\in\mathcal{F}_I$, respectively, by
\begin{align}
   \{v_h\}:=\frac{1}{2}(v_h^{(i)}+v_h^{(j)}), &\quad \text{and}\quad\llbracket v_h \rrbracket :=v_h^{(i)} - v_h^{(j)},\\
   \intertext{ and for $F_i\in \mathcal{F}_B$}
   \label{LMMT_3.16c}
  \{v_h\}:=v_h^{(i)}, &\quad\text{and}\quad \llbracket v_h\rrbracket := v_h^{(i)}.
\end{align}

The dG-IgA scheme 
reads  as follows: find $u_h\in \mathbb{B}_h(\mathcal{T}_H(\Omega))$ such that
\begin{flalign}\label{LMMT_8a}
 a_h(u_h,v_h)=&l(v_h)+p_D(u_D,v_h),{\ } \forall v_h \in \mathbb{B}_h(\mathcal{T}_H(\Omega)),
 \intertext{where}
 \label{LMMT_8b}
 a_h(u_h,v_h) = & \sum_{i=1}^N a_i(u_h,v_h)-\sum_{i=1}^N \Big(\frac{1}{2}s_i(u_h,v_h) - p_i(u_h,v_h)\Big),
 \intertext{ with  the  bilinear forms} 
\nonumber
 a_i(u_h,v_h) =& \int_{\Omega_i}\alpha\nabla u_h\nabla v_h\,dx, \\
 \nonumber
  s_i( u_h,v_h)=& \sum_{F_{ij}\subset \partial \Omega_i}\int_{F_{ij}} \{\alpha\nabla u_h\}\cdot \mathbf{n}_{F_{ij}} \llbracket v_h \rrbracket +
                                                 \{\alpha\nabla v_h\}\cdot \mathbf{n}_{F_{ij}} \llbracket u_h \rrbracket\,ds, \\
  \nonumber
 p_i(u_h,v_h) =&
                   \begin{cases}
 	               \sum_{F_{ij}\subset \partial \Omega_i} \int_{F_{ij}} \mu \Big(\frac{ \alpha^{(j)}}{h_j}+\frac{ \alpha^{(i)}}{h_i}\Big) \llbracket  u_h \rrbracket \llbracket v_h \rrbracket\,ds, & \text{if} {\ }F_{ij}\in \mathcal{F}_I\\
 	               \sum_{F_{i}\subset \partial \Omega_i} \int_{F_{i}} \mu \frac{ \alpha^{(i)}}{h_i} \llbracket  u_h \rrbracket \llbracket v_h \rrbracket\,ds,&\text{if} {\ }F_{i}\in \mathcal{F}_B\\
                   \end{cases}\\
 \nonumber
 p_D(u_D,v_h)=&   \sum_{F_{i}\subset \partial \Omega_i}\int_{F_{i}} \mu \frac{\alpha^{(i)}}{h_i}  u_D  v_h\,ds,{\ }{\ }F_{i}\in \mathcal{F}_B
\end{flalign}
where  $\alpha^{(i)}:=\alpha|_{\Omega_i}$ and
the unit normal vector  $\mathbf{n}_{F_{ij}}$ is oriented from $\Omega_i$ towards the interior of $\Omega_j$
and the  parameter $\mu>0$
will be specified later in the error analysis, cf. \cite{LMMT_Dryja:2003a}.
\par
For notation convenience in what follows, we will use the following  expression 
$$
 \int_{F_{ij}} \mu \big(\frac{\alpha^{(j)}}{h_j}+\frac{\alpha^{(i)}}{h_i}\big) \llbracket  u_h \rrbracket \llbracket v_h \rrbracket\,ds,
$$
for both cases, $F_{ij}\in \mathcal{F}_I$ and $F_{i}\in \mathcal{F}_B$. 
In the latter case, we will assume that $\alpha^{(j)}=0$.
%
\subsection{Auxiliary Results}
\label{LMMT_subsec:2.2:AuxiliaryResults}
We will use the following  auxiliary results which have been shown in \cite{LMMT_LangerToulopoulos:2014a}.

\begin{lemma}\label{LMMT_lemma1}
Let $u\in W^{l,p}(\mathcal{T}_H(\Omega))$ with $  l\geq 2$ and $ p>1$. Then there is a positive constant
$C$, which is determined according to the $C_m$ and $C_M$ of (\ref{LMMT_2_c}),  such  that

\begin{align}\label{LMMT_13b}
   \int_{{F_{ij}}\subset \partial \Omega_i}|u|^p\,ds \leq C \Big( h_{i}^{-1}\int_{\Omega_i} |u|^p\,dx +
   h_{i}^{p-1}\int_{\Omega_i}|\nabla u|^p\,dx\Big).                                                       
\end{align}
\end{lemma}
\begin{lemma}\label{LMMT_lemma2}
 For all $v_h\in {\mathbb{B}}^{(i)}_{h_i}$ defined on  $ T^{(i)}_{h_i,{\Omega_i}}$, there is a 
 positive constant
 $C$, depending on the mesh quasi-uniformity parameters and $C_m$ and $C_M$ of (\ref{LMMT_2_c})  but not on $h_i$, such that
 \begin{align}\label{LMMT_14}
  \|\nabla v_h \|^p_{L^p({\Omega_i})} \leq C h_i^{-p}\| v_h \|^p_{L^p({\Omega_i})}.
 \end{align}
\end{lemma}
\begin{lemma}\label{LMMT_lemma3}
 For all $v_h\in {\mathbb{B}}^{(i)}_{h_i}$ defined on  $ T^{(i)}_{h_i,{\Omega_i}}$ and for all ${F}_{ij}\subset \partial {\Omega_i}$,
 there is a positive constant $C$, which depends on the mesh quasi-uniformity parameters $C_m$ and $C_M$ of (\ref{LMMT_2_c})  but not on $h_i$, such that
 \begin{align} \label{LMMT_18}
   \|v_h \|^p_{L^p({F_{ij}})} \leq  C h_i^{-1}\| v_h \|^p_{L^p({\Omega_i})}.
 \end{align}
\end{lemma}

\begin{lemma}\label{LMMT_lemma4_a}
 Let $v_h\in {\mathbb{B}}^{(i)}_{h_i}$  such that
 $v_h \in W^{l,p}({E})\cap W^{m,q}({E}), {\ } {E}\in   T^{(i)}_{h_i,{\Omega_i}}$,
 and  $0\leq m\leq l, {\ }1\leq p,q \leq \infty$.
 Then there is a positive constant
 $C$, depended on the mesh quasi-uniformity parameters $C_m$ and $C_M$ of (\ref{LMMT_2_c}) but not on $h_i$, such that
 \begin{align}\label{LMMT_4.13}
  | v_h |_{W^{l,p}({E})} \leq  C  h^{m-l-\frac{d}{q}+\frac{d}{p}}_{i} |  v_h |_{W^{m,q}({E})}.
 \end{align}
\end{lemma}
\subsection{Analysis of the dG IgA Discretization}
\label{LMMT_subsec:2.3:AnalysisofthedG IgADiscretization}
Next,  we study the convergence estimates of the method (\ref{LMMT_8a}) under the following regularity assumption
for the  weak solution   
 $u\in W^{1,2}(\Omega) \cap W^{l,p}(\mathcal{T}_H(\Omega))$ with $l\geq 2$ and 
 $p\in (\max\{1,\frac{2d}{d+2(l-1)}\},2]$.  
For simplicity of the presentation, we assume that  $l\leq k+1$.
Nevertheless, for the case of highly smooth solutions,
the estimates given below, see Lemma \ref{LMMT_lemma5.3},  can be expressed 
in terms of the underlying polynomial degree $k$. More precisely, the estimate $\delta(l,p,d)=l+( d/2- d/p-1)$,
must  be replaced by $\delta(l,p,d)=\min\{l+( d/2- d/p-1, k\}$.
  We use the enlarged space 
$W_h^{l,p}:= W^{1,2}(\Omega) \cap W^{l,p}(\mathcal{T}_H(\Omega))+\mathbb{B}_h(\mathcal{T}_H(\Omega))$, 
and
  will show that the dG IgA method converges in optimal rate with respect to $\|.\|_{dG}$ norm
\begin{equation}\label{LMMT_20}
 \|u\|^2_{dG} = \sum_{i=1}^N\Big(\alpha^{(i)}\|\nabla u^{(i)}\|^2_{L^2(\Omega_i)} +
                p_i(u^{(i)},u^{(i)}) \Big), {\ }u\in  W_h^{l,2}.
\end{equation}
For the error analysis, it is necessary  to show the continuity and coercivity properties of the 
bilinear form $a_h(.,.)$ of (\ref{LMMT_8b}) and interpolation estimates in $\|.\|_{dG}$ norm.
We start by  providing  estimates on how well the quasi-interpolant $\Pi_hu$  approximates 
$u\in W^{l,p}(\Omega_i)$, see proof in \cite{LMMT_LangerToulopoulos:2014a}.
\begin{lemma}\label{LMMT_lemma5.3}
 Let $u\in W^{l,p}(\Omega_i)$ with 
 $l\geq 2$ and $ p\in (\max\{1,\frac{2d}{d+2(l-1)}\},2]$
 and let 
  $E={\Phi}_i({E}), {E}\in T^{(i)}_{h_i,{\Omega}}$. Then for  $0 \leq m \leq l \leq k+1$, 
  there exist 
 constants $C_i:=C_i\big(\max_{l_0 \leq l}(\|D^{l_0}{\Phi}_i\|_{L^{\infty}(\Omega_i)}),
 \|u\|_{W^{l,p}(\Omega_i)}\big)$, such that
 \begin{align}\label{LMMT_5.11}
 \sum_{E\in T^{(i)}_{h_i,\Omega_i}} |u-\Pi_h u|^p_{W^{m,p}(E)} \leq C_i h_i^{p(l-m)} . 
\end{align}
 Moreover,  we have the following estimates for $F_{ij}= \partial \Omega _i \cap \partial \Omega_j$:
 \begin{align}\nonumber
 (i)&{\ } h_i^{\beta} \|\nabla u^{(i)}-\nabla \Pi_h u^{(i)}\|^p_{L^{p}(F_{ij})} \leq  C_i C_{d,p} h_i^{p(l-1)-1+\beta},  \\
  \nonumber
 (ii)&{\ } \big(\frac{\alpha^{(j)}}{h_j}+\frac{\alpha^{(i)}}{h_i}\big)
 \|\llbracket u-\Pi_h u\rrbracket \|^2_{L^{2}(F_{ij})} \leq \\  
  \nonumber
 &\quad C_i\alpha^{(j)}\frac{h_i}{h_j}\Big( h_i^{\delta(l,p,d)}\|u\|^p_{W^{l,p}(\Omega_i)}\Big)^{2} + 
  C_j\alpha^{(i)}\frac{h_j}{h_i}\Big( h_j^{\delta(l,p,d)}\|u\|_{W^{l,p}(\Omega_j)}\Big)^{2}+\\
  \nonumber
& \qquad  C_j\Big( h_j^{\delta(l,p,d)}\|u\|_{W^{l,p}(\Omega_j)}\Big)^{2} + 
  C_i\Big( h_i^{\delta(l,p,d)}\|u\|_{W^{l,p}(\Omega_i)}\Big)^{2}, \\
  \nonumber
(iii)&{\ }  \|u-\Pi_h u\|^2_{dG} \leq   \sum_{i=1}^N C_i\Big( h_i^{\delta(l,p,d)}\|u\|_{W^{l,p}(\Omega_i)}\Big)^{2}  +\\
\nonumber
 &\qquad \qquad \sum_{i=1}^N \sum_{F_{ij}\subset \partial \Omega_i}C_i\alpha^{(j)}\frac{h_i}{h_j}\Big( h_i^{\delta(l,p,d)}\|u\|_{W^{l,p}(\Omega_i)}\Big)^{2},
 \end{align}
 where $\delta(l,p,d)= l+( d/2- d/p-1)$. 
\end{lemma}
We mention that the proof of estimate $(iii)$ in Lemma \ref{LMMT_lemma5.3} can be derived
by using the estimates $(i)$,  $(ii)$ and Lemma \ref{LMMT_lemma1}. 
\begin{lemma}\label{LMMT_lemma5}
  Suppose $u_h\in \mathbb{B}_h(\cal{S}(\Omega))$.  
 There exist a positive constant $C$, independent of $\alpha$ and  $h_i$, such that
 \begin{align}\label{LMMT_22}
  a_h(u_h,u_h) \geq C \, \|u_h\|_{dG}^2,{\ }\; \forall \,u_h\in \mathbb{B}_h(\cal{S}(\Omega)).
 \end{align}
\end{lemma}
\begin{proof}
 By (\ref{LMMT_8a}), we have that
 \begin{align}\label{LMMT_24}
 \nonumber
  a_h(u_h,u_h)  =  \sum_{i=1}^N  a_i(u_h,u_h)-\frac{1}{2}
                   \sum_{i=1}^N s_i(u_h,u_h)+p_i(u_h,u_h) =\\
  \nonumber
                   \sum_{i=1}^N \alpha_i\|\nabla u_h\|^2_{L^2(\Omega_i)} - 
                  2 \sum_{F_{ij}\in \mathcal{F}} \int_{F_{ij}} \{\alpha\nabla u_h\}\cdot \mathbf{n}_{F_{ij}} \llbracket u_h\rrbracket \,ds \\
                 +  \sum_{F_{ij}\in \mathcal{F}} \mu\Big(\frac{\alpha^{(i)}}{h_i}+\frac{\alpha^{(j)}}{h_j}\Big)\|\llbracket u_h\rrbracket\|^2_{L^2(F_{ij})}.
 \end{align}
For the second term on the right hand side,  Lemma \ref{LMMT_lemma2} and 
the trace inequality (\ref{LMMT_18})
expressed on $F_{ij}\in \mathcal{F}$ yield the bound
\begin{multline}\label{LMMT_25}
  -\sum_{F_{ij}\in \mathcal{F}} \int_{F_{ij}} \{\alpha\nabla u_h\}\cdot \mathbf{n}_{F_{ij}} \llbracket u_h \rrbracket \,ds \geq \\
  \hspace*{-10mm}{
 - C_{1,\varepsilon}\sum_{i=1}^N \alpha_i\|\nabla u_h\|^2_{L^2(\Omega_i)} - 
    \sum_{F_{ij}\in \mathcal{F}}\frac{1}{C_{2,\varepsilon}}\Big(\frac{\alpha^{(i)}}{h_i}+\frac{\alpha^{(j)}}{h_j}\Big)\|\llbracket u_h \rrbracket\|^2_{L^2(F_{ij})}.
}
   \end{multline}
Inserting (\ref{LMMT_25}) into (\ref{LMMT_24}) 
and choosing $  C_{1,\varepsilon} < 1/2$ and $\mu > 2/C_{2,\varepsilon}$,
we obtain (\ref{LMMT_22}). \hfill$\Square$
\end{proof}
Due to assumed regularity of the solution,
the normal interface fluxes  $(\alpha \nabla u)|_{\Omega_i}\cdot \mathbf{n}_{F_{ij}}$ belongs (in general)  to $L^p(F_{ij})$.
The following   bound for the interface fluxes in $\|.\|_{L^p}$ setting has been shown in \cite{LMMT_LangerToulopoulos:2014a}.
\begin{lemma}\label{LMMT_lemma5.1}
 There is a constant $C$ such that the following inequality for 
 $(u,v_h)\in W_h^{l,p}\times \mathbb{B}_h(\cal{S}(\Omega))$ 
 holds true
 \begin{align}\label{LMMT_5.2}
  \sum_{i=1}^N\sum_{F_{ij}\subset \partial \Omega_i}\int_{F_{ij}}\{\alpha\nabla u\}\cdot \mathbf{n}_{F_{ij}}\llbracket v_h \rrbracket \,ds \leq& \\ 
  \nonumber
   C\Big(\sum_{F_{ij}\in \mathcal{F}}
   \alpha^{(i)}h_i^{1+\gamma_{p,d}}\| \nabla u^{(i)}\|^p_{L^p(F_{ij})}+&
   \alpha^{(j)} h_j^{1+\gamma_{p,d}}\| \nabla u^{(j)}\|^p_{L^p(F_{ij})}\Big)^{\frac{1}{p}} \|v_h\|_{dG},\\
  \nonumber 
   \text{where}{\ }\gamma_{p,d} = d(p-2)/2.\hskip 2cm\qquad&
 \end{align}
\end{lemma}
\begin{proof}
	We use  H\"older's inequality and then the results of 
	Lemmas \ref{LMMT_lemma3} and  \ref{LMMT_lemma4_a}. \hfill$\Square$
\end{proof}

Applying similar procedure as this  in  Lemma \ref{LMMT_lemma5.1}, 
we can show for the symmetrizing terms that there is
a positive constant independent of grid size such that
 \begin{align}\label{LMMT_5.2_a}
\sum_{i=1}^N\sum_{F_{ij}\subset \partial \Omega_i}\int_{F_{ij}}\{\alpha\nabla v_h\}\cdot \mathbf{n}_{F_{ij}} \llbracket u \rrbracket\,ds 
\leq C_1 \|v_h\|_{dG} \|u\|_{dG}. 
\end{align}
Using the results (\ref{LMMT_5.2}) and (\ref{LMMT_5.2_a}),  we can  show the boundedness of the 
bilinear form, see details in \cite{LMMT_LangerToulopoulos:2014a}.
\begin{lemma}\label{LMMT_lemma5.2}
 There is  a $C$ independent of $h_i$ such that for $ (u,v_h)\in W^{l,p}_h\times \mathbb{B}_h(\cal{S}(\Omega))$
 \begin{align}\label{LMMT_5.8}
  a_h(u,v_h)\leq C (\|u\|_{dG}^p + \sum_{F_{ij}\in \mathcal{F}}
   h_i^{1+\gamma_{p,d}}\alpha^{(i)}\| \nabla u^{(i)}\|^p_{L^p(F_{ij})}+ \\
   \nonumber
   h_j^{1+\gamma_{p,d}}\alpha^{(j)}\| \nabla u^{(j)}\|^p_{L^p(F_{ij})}\Big)^{\frac{1}{p}} \|v_h\|_{dG} .
 \end{align}
\end{lemma}

Next,  we give the main error estimate for the dG IgA method. 
\begin{theorem}
\label{LMMT_:sec:2:th1:dGnormErrorEstimate}
 Let $u\in W^{1,2}(\Omega) \cap W^{l,p}(\mathcal{T}_H(\Omega))$ 
 be the solution of (\ref{LMMT_sec:2:eqn:Model1-VF}), 
 with $l \geq 2$ and $p \in (\max\{1,\frac{2d}{d+2(l-1)}\},2]$.
 Let $u_h\in \mathbb{B}_h(\cal{S}(\Omega))$ be the dG IgA solution of (\ref{LMMT_8a}).  
Then there are 
 $C_i:=C_i\Big(\max_{l_0 \leq l} \Big(\|D^{l_0}{\Phi}_i\|_{L^{\infty}(E)}\Big),
 \|u\|_{W^{l,p}(\Omega_i)}  \Big)$, such that
 \begin{align}\label{LMMT_5.24}
   \|u-u_h \|_{dG} \leq &  \sum_{i=1}^N\Big( C_i\Big( h_i^{\delta(l,p,d)} +
 \sum_{F_{ij}\subset \partial \Omega_i}\alpha^{(j)}\frac{h_i}{h_j}  
  h_i^{\delta(l,p,d)}\Big)\|u\|_{W^{l,p}(\Omega_i)}\Big),
 \end{align}
 where $\delta(l,p,d)= l+(d/2-d/p-1)$.
\end{theorem}
\begin{proof}
  First we need to prove  the consistency of $u$, i.e.  $u$ satisfies (\ref{LMMT_8a}). 
   Then, we use a variation of Cea's Lemma (expressed in the dG framework), the results
   of Lemmata \ref{LMMT_lemma5} and \ref{LMMT_lemma5.2}, 
   as well the quasi-interpolation estimates of Lemma \ref{LMMT_lemma5.3}.
   A complete proof can be found in \cite{LMMT_LangerToulopoulos:2014a}. 
   $\Square$
\end{proof}

%
\subsection{Numerical Examples}
\label{LMMT_subsec:2.4:NumericalExamples}
In this section, we present a series of numerical 
examples to validate the theoretical results, which have been presented. 

\subsubsection{Smooth and Low-regularity Solutions}
We restrict ourselves to a model problem in $\Omega=(-1,1)^3$, 
with $\Gamma_D=\partial \Omega$. 
The domain $\Omega$ is subdivided into four equal subdomains $\Omega_i$, $i=1,...,4$,
 where for simplicity every $\Omega_i$ is initially partitioned into a mesh 
 $T^{(i)}_{h_i,\Omega_i}$,  with  $h:=h_i=h_j, i\neq j$, and $i,j=1,...,4$.
 Successive uniform refinements are performed 
 on every $T^{(i)}_{h_i,\Omega_i}$ in order to compute numerically
 the convergence rates. 
We set the diffusion coefficient equal to one.
\par
In the first test, 
the data $u_D$ and $f$ in (\ref{LMMT_sec:1:eqn:Model1}) are determined such that the exact solution  is 
$u(x)=\sin(2.5\pi x)\sin(2.5\pi y)\sin(2.5\pi z)$ (smooth test case).
The first two columns of Table \ref{LMMT_table_1} display the convergence rates. 
As it was expected, the convergence rates are optimal.  
In the second case, the exact solution is 
$u(x)=|x|^{\lambda}$. The parameter $\lambda$ is chosen such that $u\in W^{l,p=1.4}(\Omega)$.
Specifically, for $l=2$, we set $\lambda=0.58$ and for $l=3$, we set $\lambda=1.58$. 
In the  last columns of Table \ref{LMMT_table_1},
we display the convergence rates for degree $k=2$ and $k=3$ in case of having  $l=2$ and $l=3$.
We observe that, for each of the two different tests, the error in the 
dG-norm behaves according to the main error estimate given by (\ref{LMMT_5.24}).

\begin{table}[th!]
\centering %
\caption{The numerical convergence rates of the dG IgA method.}
\label{LMMT_table_1}
\begin{tabular}{|c| c|c|c|c|c|c|} 
\hline 
          &\multicolumn{2}{|c|}{highly smooth }
          &\multicolumn{2}{|c|}{$k =2$ } &\multicolumn{2}{|c|}{$k =3$ }\\ [0.5ex] \hline 
$\frac{h}{2^s}$    &$k=2$& $k=3$&  $l=2$ & $l=3$        &  $l=2$ & $l=3$  \\ \hline 
           \multicolumn{7}{|c|} {Convergence rates } \\ [0.5ex] \hline
 $s=0$  &-    & -     & -     &-      & -      & -  \\ 
 $s=1$  &0.15 &2.91   & 0.62  &0.76   & 0.24   &1.64    \\ 
 $s=2$  &2.34 &2.42   & 0.29  &1.10   & 0.28   &0.89    \\ 
 $s=3$  &2.08 &3.14   & 0.35  &1.32   & 0.47   &1.25    \\ 
 $s=4$  &2.02 &3.04   & 0.35  &1.36   & 0.36   &1.37    \\ \hline
\end{tabular}
\end{table}


\subsubsection{Non-Matching Meshes} 
\label{LMMT_subsubsec:3.x.y:NonMatching}
We consider the  boundary value problem (\ref{LMMT_sec:1:eqn:Model1})
with exact solution $u(x,y)= \sin (\pi x)$\\$ \sin(\pi y)$. 
The computational domain consist of two unit square patches, 
$\Omega_1=(-1,0)\times(0,1)$ and $\Omega_2=(0,1)\times (0,1)$, 
see Fig.~\ref{LMMT_fig:sec:SquareNonMatching} (left).
The knot vectors representing the geometry are given by
$\Xi_1^{(1,2)}=\Xi_2^{(1,2)}= \{0, 0, 1, 1\}$.  
We refine the mesh of the patches to a ratio $R$, 
i.e.,
the ratio of the
grid sizes is 
$R={h_1}/{h_2}$, 
see Fig.~\ref{LMMT_fig:sec:SquareNonMatching} (left).
We solved the problem using equal B-spline degree $k$ on all patches. We plot in  
 Fig.~\ref{LMMT_fig:sec:SquareNonMatching} (right) the dG IgA solution $u_h$ computed with $k=2$.
 In Fig.~\ref{LMMT_fig:sec:ErrorSquareNonMatching},
we present the decay  of the $L_2$ and $dG$ errors
for $k=1,2,3$ and ratios $R$ from 1 up to 5.
In Table \ref{LMMT_tab:sec:Square2PatchesRefine40}, we
 display the convergence rates for large ratio $R=40$. 
In both cases, we observe that the rates are the expected 
 one and are not affected by the different grid sizes of the meshes.
\begin{figure}[th!]
 \centering
  \includegraphics[width=0.49\textwidth, height = 0.20\textheight]{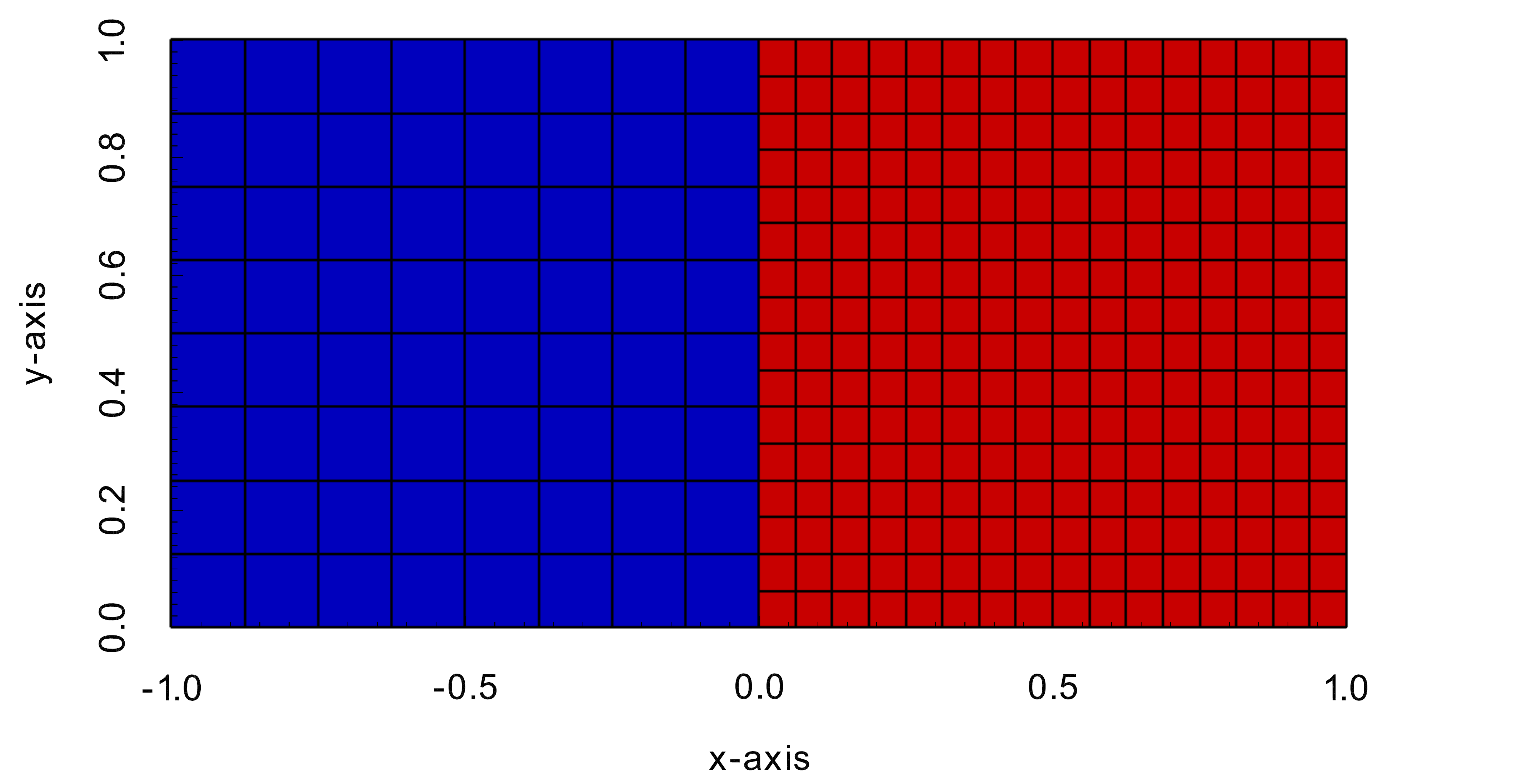}
  \includegraphics[width=0.49\textwidth, height = 0.20\textheight]{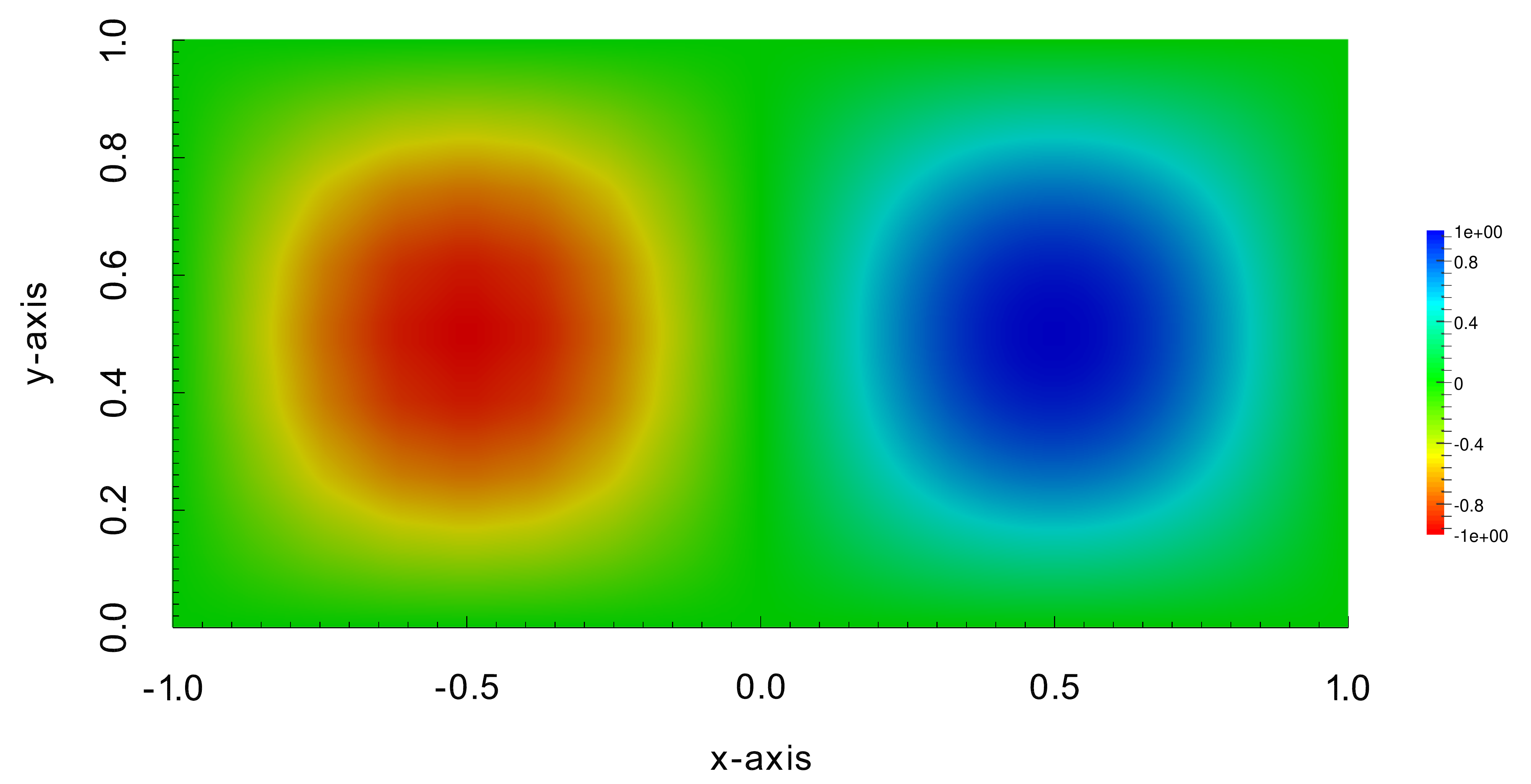}
  \caption{Non-matching meshes: Decomposition into 2 patches with underlying grid 
           of ratio $R = 2$ (left), contours  of $u_h$ (right).
	  }
  \label{LMMT_fig:sec:SquareNonMatching}
\end{figure} 
%
\begin{figure}[th!]
 \centering
  \includegraphics[width=0.49\textwidth, height = 0.22\textheight]{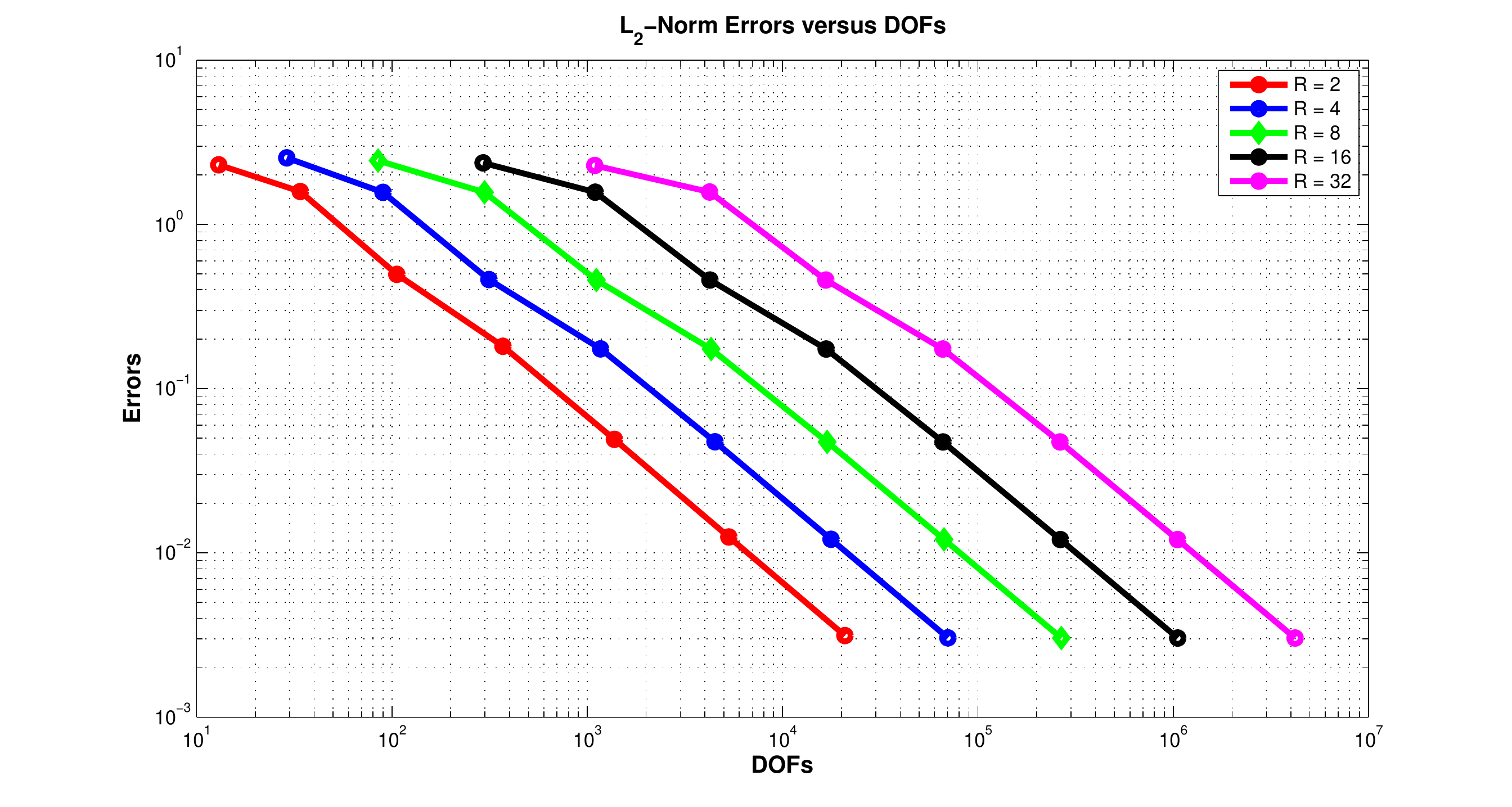}
  \includegraphics[width=0.49\textwidth, height = 0.22\textheight]{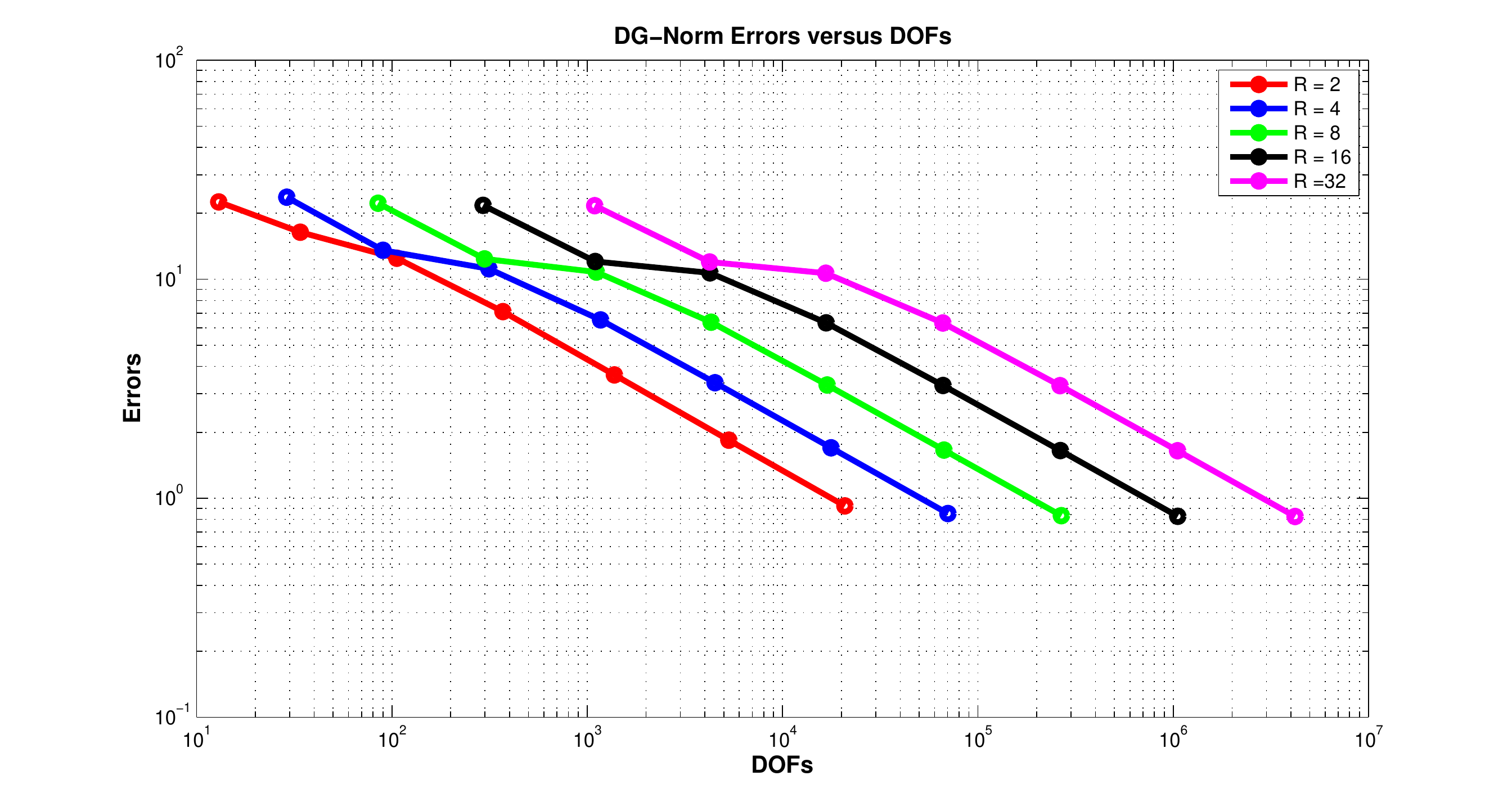}
  \includegraphics[width=0.49\textwidth, height = 0.22\textheight]{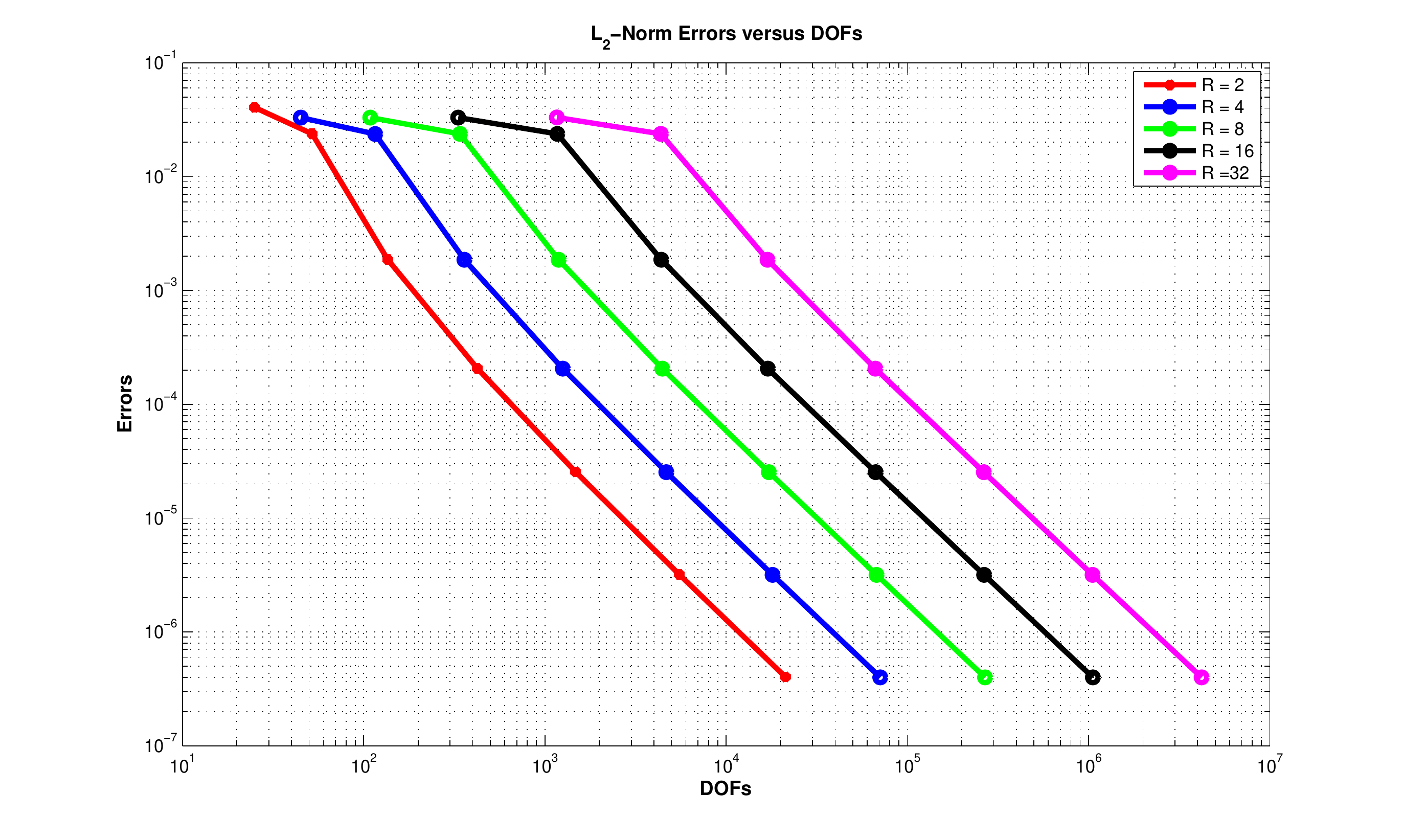}
  \includegraphics[width=0.49\textwidth, height = 0.22\textheight]{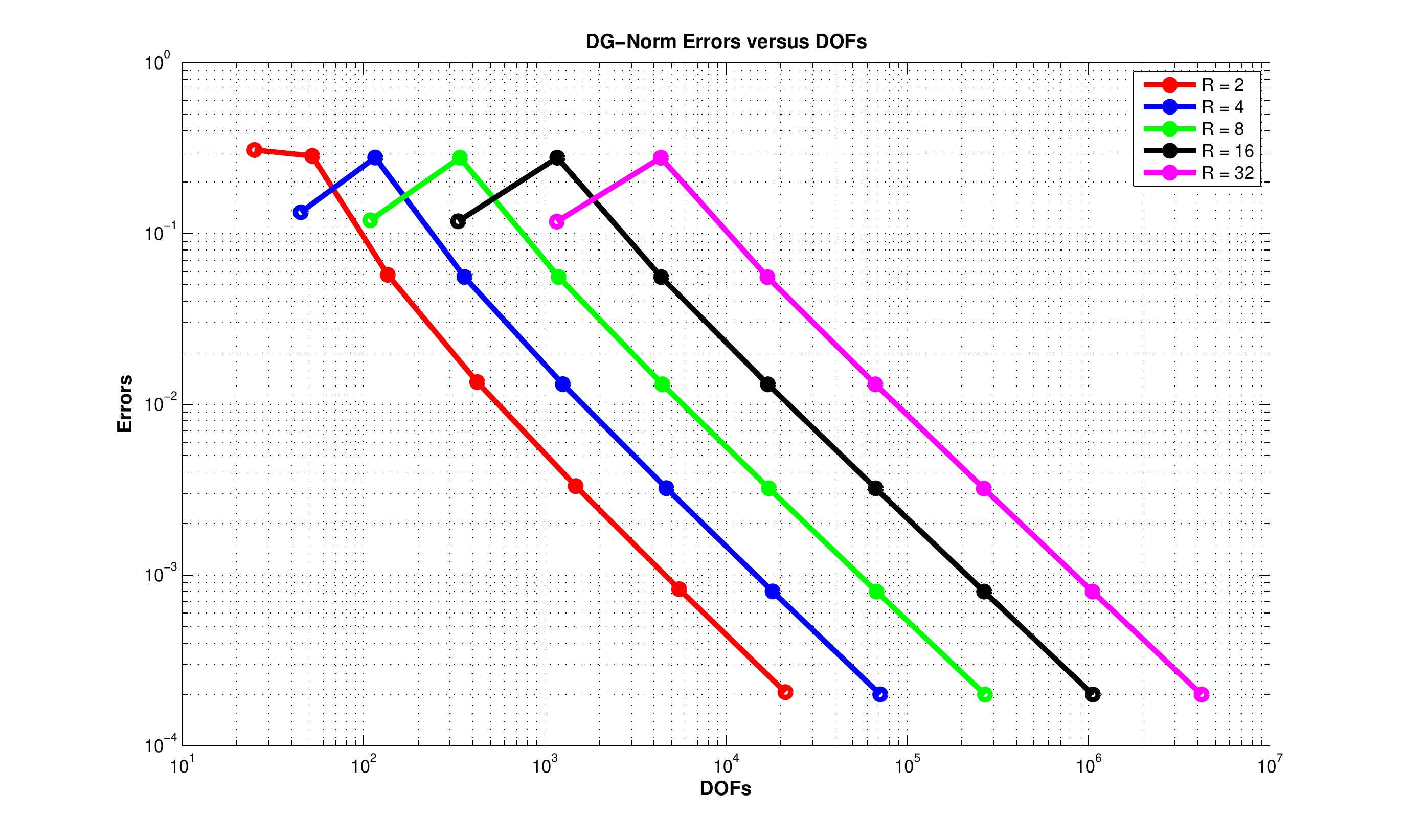}
  \includegraphics[width=0.49\textwidth, height = 0.22\textheight]{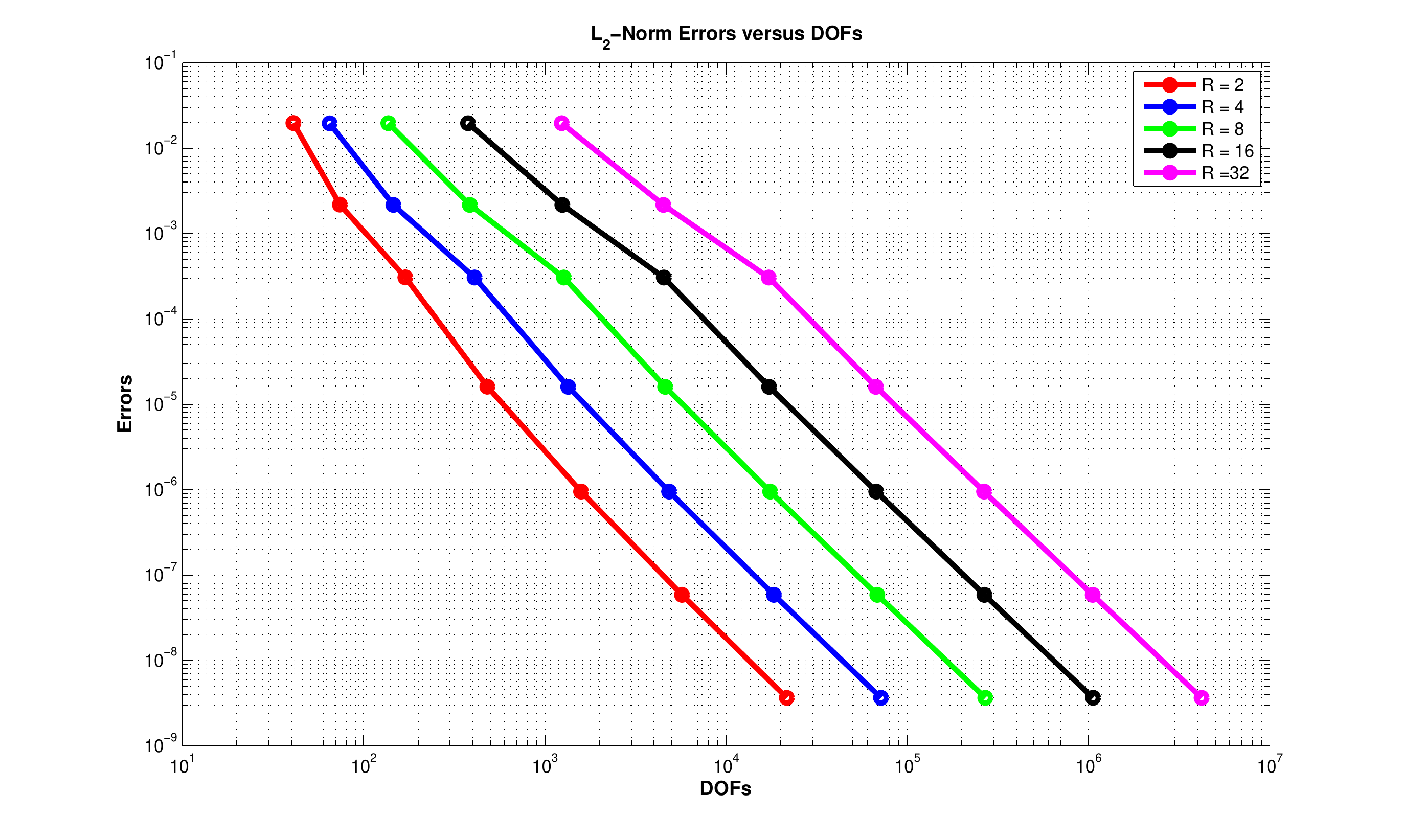}
  \includegraphics[width=0.49\textwidth, height = 0.22\textheight]{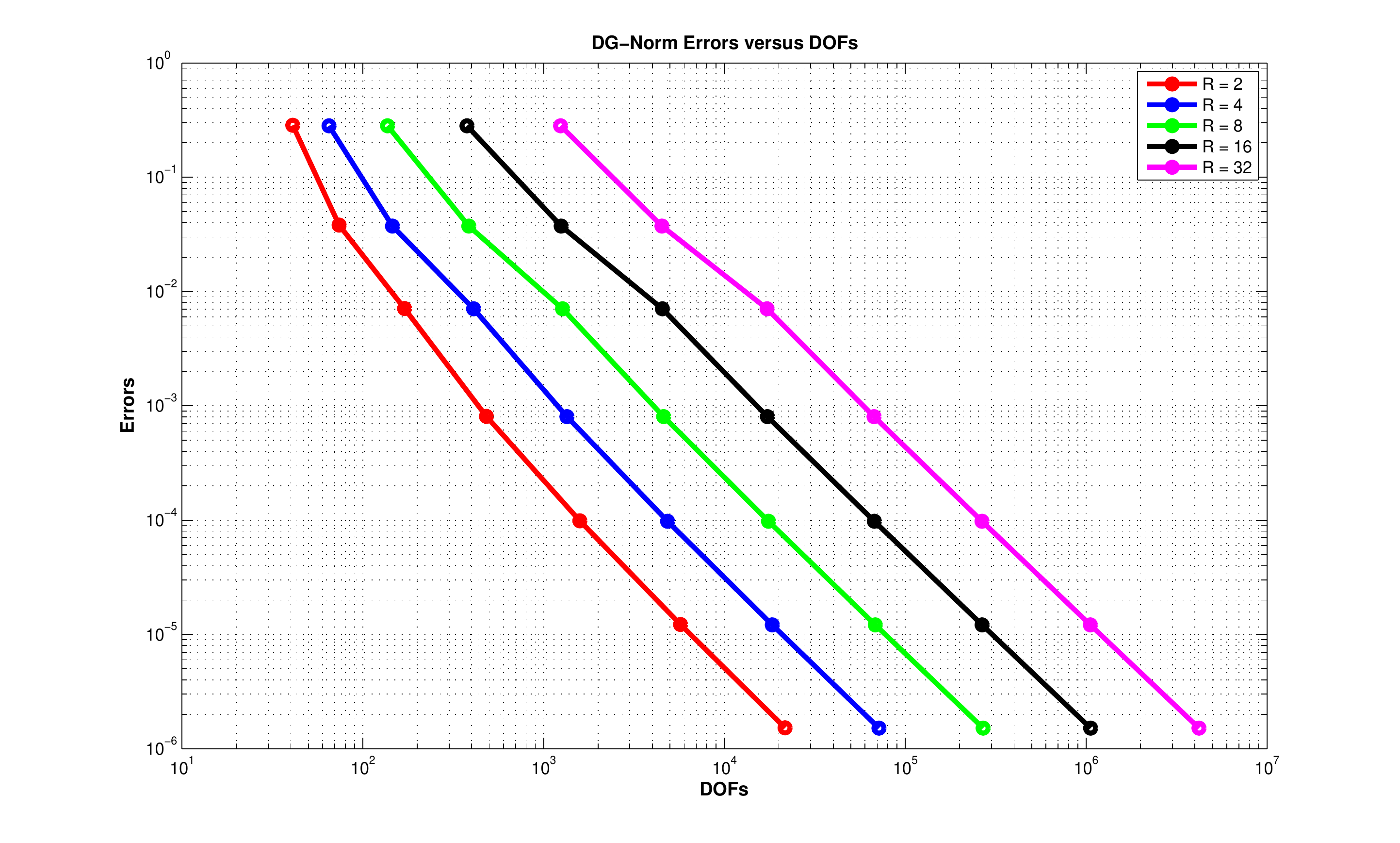}
  \caption{Error decay in the $L_2$ (left) and dG (right) norms for polynomial 
		  degree $k = 1,2,3$ (top to bottom) and ratio $R = 2^r$ with $r$ running from 1 to 5.
	  }
	 \label{LMMT_fig:sec:ErrorSquareNonMatching}  
\end{figure} 
\begin{table}[th!]
 \centering  
 \caption{Non-Matching Meshes: error estimates for degrees $k = 1,2,3$ and $R = 40.$}
   \label{LMMT_tab:sec:Square2PatchesRefine40}
  \begin{tabular}{|l|l|l|l|l|l|l|l|}
   \hline
     \cline{2-5}   
       Dofs                &$L_{2}$ error  &conv. rate   &DG error     &conv. rate  \\  \hline\hline
       \multicolumn{5}{|c|}{Degree $k=1$} \\ \hline
       1685                &0.31667       &0         &2.1596          &0	\\
       6570                &0.0867936     &1.86732   &0.97857         &1.14202	\\
      25946                &0.0249492     &1.79859   &0.49992         &0.968979	\\
     103122                &0.00638919    &1.96529   &0.251392        &0.991758	\\
     411170                &0.00160452    &1.99349   &0.125887        &0.997804	\\
1.64205$\times 10^{6}$     &0.000401493   &1.9987    &0.0629683       &0.999435	\\
6.56295$\times 10^{6}$     &0.000100393   &1.99972   &0.0314873       &0.999857	\\  \hline
       \multicolumn{5}{|c|}{Degree $k=2$} \\ \hline
       1773                   &0.0330803                    &0         &0.117511        &0           \\
       6740                   &0.0237587                    &0.477516  &  0.278058      &-1.2426     \\
      26280                   &0.00186298                   &3.67278   &0.0555191       &2.32433     \\
     103784                   &0.000205886                  &3.17769   &0.0130802       &2.0856      \\
     412488      	      &2.53126$\times 10^{-5}$      &3.02392   &0.00321704      &2.02358     \\
1.64468$\times 10^{6}$        &3.17795$\times 10^{-6}$      &2.99369   &0.000800278     &2.00716     \\
 6.5682$\times 10^{6}$        &3.99389$\times 10^{-7}$      &2.99223   &0.000199727     &2.00247     \\ \hline

       \multicolumn{5}{|c|}{Degree $k=3$} \\ \hline
       1865                &0.0196108                   &     0    &0.281761                     &   0     \\
       6914                &0.00216825                  &3.17704   &0.037403                     &2.91324  \\
      26618                &0.00030589                  &2.82545   &0.00706104                   &2.4052   \\
     104450                &1.60210$\times 10^{-5}$     &4.25498   &0.000803958                  &3.13469  \\
     413810                &9.49748$\times 10^{-7}$     &4.07627   &9.7687$\times 10^{-5}$       &3.04088  \\
1.64731$\times 10^{6}$     &5.85542$\times 10^{-8}$     &4.0197    &1.21191$\times 10^{-5}$      &3.01089  \\
6.57346$\times 10^{6}$     &3.64710$\times 10^{-9}$     &4.00495   &1.51195$\times 10^{-6}$      &3.00280  \\  \hline
  \end{tabular}  

\end{table}
\subsection{Graded Mesh Partitions for the dG IgA Methods}
\label{LMMT_subsec:2.6:GradedMeshPartitionsforthedG IgAMethods}
We saw in the previous numerical tests that the presence of singular points reduce the convergence 
rates. In this section, we will study this subject in a more general form. 
We will focus on  solving the model problem in
domains with re-entrant corners on the boundary. 
Due to these singular corner points, the regularity of the solution (at least in a small 
vicinity) is reduced in comparison with the solutions in smooth domains  
\cite{LMMT_Grisvard:1985a, LMMT_Grisvard:2011a}. As a result, the numerical methods applied on 
quasi uniform meshes for solving these problems  do not yield  the optimal convergence rate and thus
a particular treatment must be applied. We will devise the popularly known graded mesh 
techniques which have widely been applied so far for finite element  methods
\cite{LMMT_ApelSandingWhiteman:1996, LMMT_ApelMilde1996}.

\subsubsection{Regularity properties  of the solution around the  boundary singular points}
Let us assume a domain $\Omega\subset \mathbb{R}^2$ and let $P_s\in\partial \Omega$ be a boundary point
with internal angle  $\omega \in (\pi,2\pi)$. 
  We consider the local cylindrical coordinates $(r,\theta)$ with pole $P_s$,
  and define the cone, see Fig.~\ref{LMMT_fig:1} (left), 
  \begin{align}\label{LMMT_gm0.1}
   \hspace*{-5mm}{
   \mathcal{C}=\{(x,y)\in \Omega: x=r\cos(\theta),y=r\sin(\theta), 0<r<R,0<\theta<\omega\}.
   }
  \end{align}
Then the solution in $\mathcal{C}$ can be written, \cite{LMMT_Grisvard:1985a},
\begin{align}\label{LMMT_gm0.4}
 u=u_{r}+u_{s},
\end{align}
where $u_r\in W^{l\geq 2,2}(\Omega)$ and 
 \begin{align}\label{LMMT_gm0.3}
   u_s=\xi(r)\gamma r^{\lambda}\sin(\lambda \theta),
  \end{align}
where $\gamma$ is the \textit{stress intensity factor} 
(is a real number depending only on $f$),
and  
$\lambda={\pi}/{\omega}\in (0,1)$ 
is
an exponent which determines the strength of the singularity. 
Since $\lambda <1$, by an easy computation, we can show that
the singular function $u_s$ does not belong to $W^{2,2}(\Omega)$ but 
$u\in W^{2,p}(\Omega)$
with $p={2}/{2-\lambda}$.
The 
representation 
(\ref{LMMT_gm0.3}) 
of $u_s$
helps us to reduce our study to the examination of the behavior
of $u$ in the vicinity of the singular point, since the  regularity properties of $u$ are determined
by the regularity of $u_s$. The main idea is the following: based on the a priori knowledge of the analytical 
form of $u_s$ in $\mathcal{C}$, we carefully construct a locally adapted mesh in $\mathcal{C}$ by
introducing a grading control parameter $\mu:=\mu(\lambda,k)$, such as
allows us to prove that the approximation order of the method applied on this adapted mesh for $u_s$ is similar
with the order of the method applied on the rest of the mesh (maybe quasi uniform) for $u_r$. 

\subsubsection{The graded mesh for $\mathcal{T}_H(\Omega)$ and global approximation estimates }
The area  $U_s:=\{x\in \Omega: |P_s-x|\leq R, R\geq N_Z h,{\ }N_Z\geq 2\}$ 
 is further  sub-divided  into   ring zones  $Z_{\zeta}, {\ }\zeta=0,..,\zeta_M < N_Z$,  with
 distance from $P_s$ equal to $D_{(Z_{\zeta},P_s)}:=C(n_\zeta h)^{\frac{1}{\mu}}$, where $1\leq  n_\zeta < N_Z $ 
 and $\frac{1}{2}\leq C \leq 1$. 
 The  radius of every zone is defined to be 
 $R_{Z_{\zeta}}:=D_{(Z_{\zeta+1},P_s)} - D_{(Z_{\zeta},P_s)}=C(n_{\zeta+1} h)^{\frac{1}{\mu}} -
 C(n_{\zeta} h)^{\frac{1}{\mu}}$.
 
 \begin{figure}[bth!]
\centering
\includegraphics[width  = 0.30\textwidth]{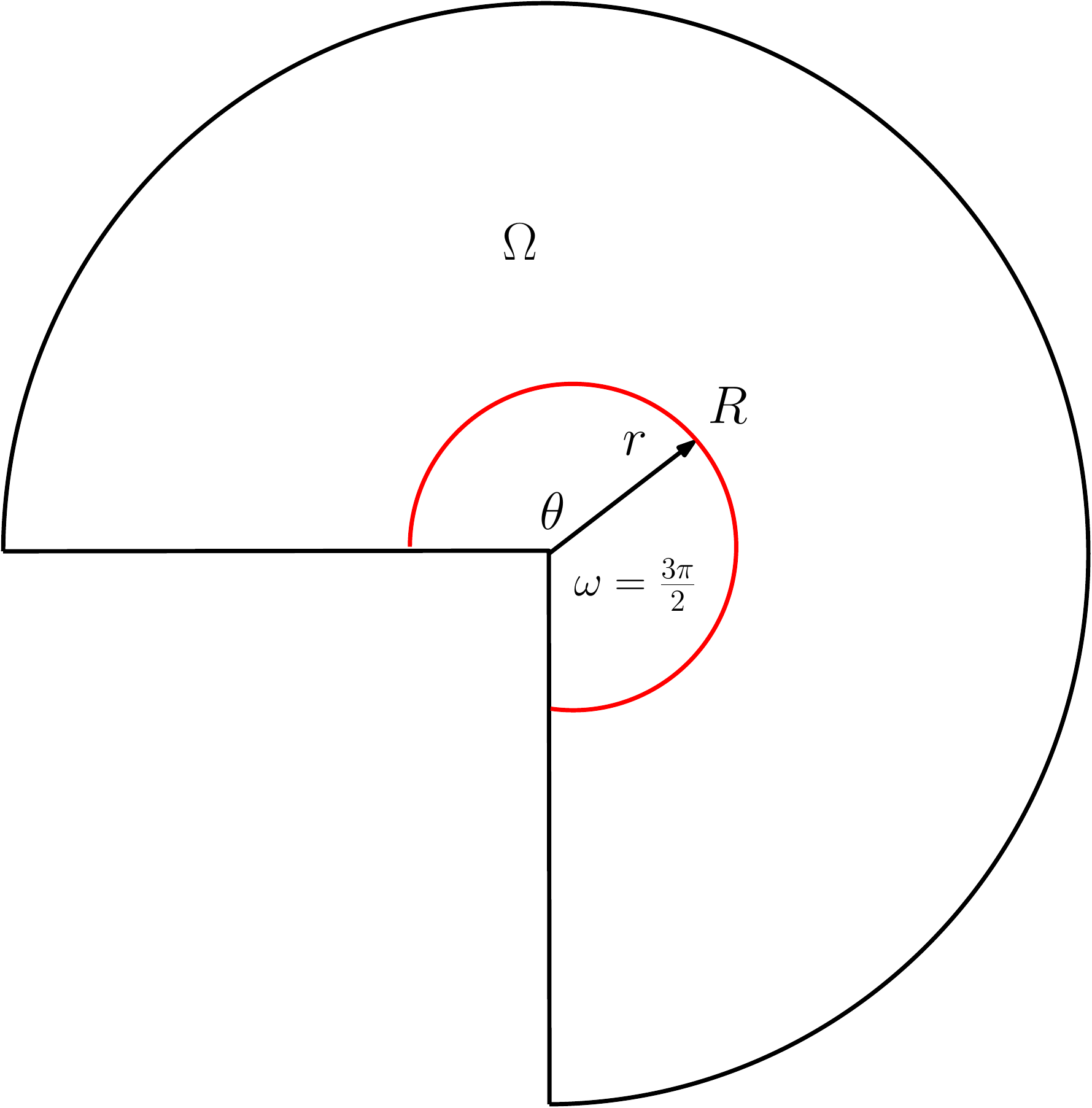} \hspace{1.5cm}
\includegraphics[width  = 0.30\textwidth]{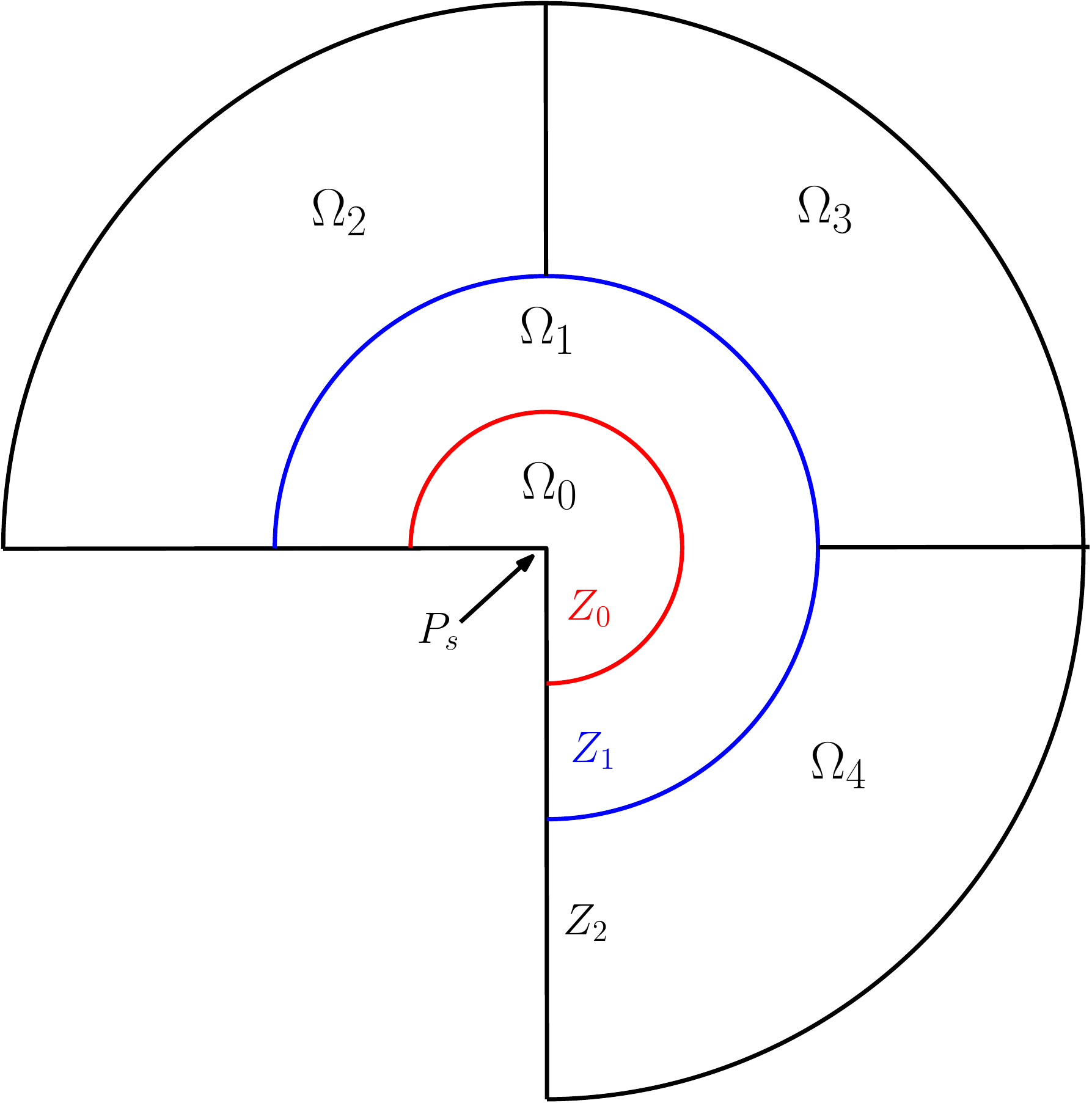}
 \caption{Left: 2d domain with corner singularity. Right: 
          Zone partition and the subdomians $\Omega_i$ of $\cal{T}_H(\Omega)$. }
\label{LMMT_fig:1}       
\end{figure}
 \par
 For convenience, we assume that the initial subdivision $\mathcal{T}_{H}$ 
 fulfill the following conditions, for an illustration see 
   Fig. \ref{LMMT_fig:1} (right) with $\zeta_M=2$:
\begin{itemize}
	\item The subdomains can be grouped into those which belong (entirely) to the area $U_s$ and those
	that belong (entirely) to $\Omega\setminus U_s$. This means that there is no $\Omega_i,{\ } i=1,...,N$
         such that $U_s\cap \Omega_i\neq \emptyset$ and $(\Omega\setminus U_s)\cap \Omega_i\neq \emptyset$.
        \item Every ring zone $Z_{\zeta}$  is partitioned into ``circular'' subdomains $\Omega_{i_{\zeta}}$ which
              have radius $\Omega_{i_{\zeta}}$ equal to the radius of the zone,
              that is $R_{\Omega_{i_{\zeta}}}= R_{Z_{\zeta}}$. For
              computational efficiency reasons, we prefer -if its possible- every 
              zone to be only represented by one subdomain. 
        \item The zone $Z_{0}$ is represented by one subdomain, say $\Omega_0$, and the mesh 
               $T^{(0)}_{h_0}(\Omega_0)$ includes all $E\in T_h(\Omega)$ such that $\partial E \cap P_s \neq \emptyset$. 
\end{itemize}
 The graded meshes  
$T^{(i_{\zeta})}_{h_{i_{\zeta}}}(\Omega_{i_{\zeta}})$  are mainly determined by 
the grading parameter $\mu(\lambda,k) \in (0,1]$ and the mesh sizes $h_{i_{\zeta}}$ are chosen to satisfy 
the following properties: 
for $\Omega_{i_{\zeta}}$  with distance $R_{\Omega_{i_{\zeta}}}$ from $P_s$,
 the mesh size  $h_{i_{\zeta}}$  
is defined to be $h_{i_{\zeta}}=\cal{O}(h R_{\Omega_{i_{\zeta}}}^{1-\mu})$ and 
for  $T^{(i_0)}_{h_0}(\Omega_0)$ the mesh size is  
of order $h_{i_{0}}=\cal{O}(h^{\frac{1}{\mu}})$, 
more details are given in \cite{LMMT_LangerMantzaflarisMooreToulopoulos:2014}.
Based on previous properties of the $T^{(i_{\zeta})}_{h_{i_{\zeta}}}(\Omega_{i_{\zeta}})$ meshes,
we can conclude the relations
\begin{align}\label{LMMT_3.2a}
 C_m h^{\frac{1}{\mu}} &\leq h_{{i_{\zeta}}}   \leq C_M h^{\frac{1}{\mu}}, &\text{if}{\ }  \overline{\Omega}_{i_{\zeta}}\cap P_s\neq \emptyset, \\
\label{LMMT_3.2b}
 C_m hR_{\Omega_{i_{\zeta}}}^{1-\mu} &        \leq h_{i_\zeta}   \leq C_M h D_{(Z_{\zeta},P_s)}^{1-\mu}, &\text{if}{\ }  \overline{\Omega}_{i_\zeta}\cap P_s= \emptyset. 
\end{align}

   Using the local interpolation estimate of  Lemma \ref{LMMT_lemma5.3} in every $\Omega_{i_{\zeta} } \subset U_s$
   and the characteristics of the meshes  $T^{(i_{\zeta})}_{h_{i_{\zeta}}}(\Omega_{i_{\zeta}})$,
   we can easily obtain  the estimate
     \begin{align}\label{LMMT_gm3.7}
 	  \|u_s-\Pi_h u_s\|_{dG(U_s)} \leq \sum_{i_{\zeta}} C_{i_{\zeta}} h_{i_{\zeta}}^{\lambda} ,
     \end{align}
         since $u_s\in W^{l=1+\lambda,2}(\Omega)$ (and  also $u\in W^{2,p=\frac{2}{2-\lambda}}(\Omega)$) and 
         we do not consider nonä-matching grid interfaces. 
         Using the mesh properties (\ref{LMMT_3.2a}) and (\ref{LMMT_3.2b}), 
          estimate   (\ref{LMMT_gm3.7}) and Lemma \ref{LMMT_lemma5.3}, we can prove 
the following global error estimate 
         of the proposed dG IgA method applied to problems with boundary singularities,
	  see \cite{LMMT_LangerMantzaflarisMooreToulopoulos:2014}.
\begin{theorem}\label{LMMT_Thrm1}
Let   $Z_{\zeta}$ be a partition of $\Omega\subset \mathbb{R}^2$ into   ring zones and 
let $\mathcal{T}_{H}$ to be a sub-division to $\Omega$ with the properties as listed in the previous paragraphs.
Let $T_{h_i}^{(i)}(\Omega_i)$ be the meshes of subdomains  as described above.   
Then for the solution $u$ of (\ref{LMMT_sec:2:eqn:Model1-VF}) and the dG IgA solution
$u_h$, we have the approximation result
\begin{align}\label{LMMT_gm3.8}
	\|u- u_h\|_{dG} \leq C h^{r}, {\ } \text{with}{\ }r=\min(k,{\lambda/ \mu}),
\end{align}
where the constant $C>0$ depends on the characteristics of the mesh and on
the mappings $\Phi_i$ (see (\ref{LMMT_2_0})) but not on $h_i$. 
\end{theorem}


\subsubsection{Numerical examples}
\label{LMMT_subsec:2.7:NumericalExamples}
In this section, we present a series of numerical examples in order to validate  the theoretical
analysis on the graded mesh in 
Section~\ref{LMMT_subsec:2.6:GradedMeshPartitionsforthedG IgAMethods}. 
The first example concerns  a two-dimensional problem with a boundary point singularity 
(L-shape domain).   The second example is the interior point singularity problem of the 
Section~\ref{LMMT_subsec:2.4:NumericalExamples}. 

\begin{enumerate}
 \item \textbf{Boundary Singular Point}\\
One of the classical test cases is the singularity due to a re-entrant corner. 
The L-shape domain given by $(-1, 1)^{2} \setminus (-1, 0)^{2}$. In  Fig.~\ref{LMMT_fig:2} (left), 
the subdivision of $\Omega$ into two subdomains is presented. The exact solution is
$u = r^\frac{\pi}{\omega}\sin(\theta \pi/\omega),$ where $\omega= 3\pi/2$. 
We set $\Gamma_D=\partial \Omega$ and the data ${f},{u}_D$ of \eqref{LMMT_sec:2:eqn:Model1-VF} 
are specified by the  given exact solution. 
The problem has been  solved  using $B$-splines of degree $k=1$ and $k=2$ and
the grading parameter is  $\mu=0.6$ and $\mu=0.3$ respectively. In  Fig. \ref{LMMT_fig:2} (middle), 
the graded mesh for $\mu=0.6$ is presented and in  Fig. \ref{LMMT_fig:2} (right) the contours of 
the numerical solution are plotted.  
In Table \ref{LMMT_ConvRate_Lshape}, we present the  convergence rates of the method without grading
(left columns). As we can see, the convergence rates  are determined by 
the regularity of the solution around the singular boundary point. In the right columns, 
we present the convergence rates corresponding to the graded meshes. 
We can see that the rates 
tend to be optimal with respect the $B$-spline degree. 

\begin{figure}[bth!]
\centering
\includegraphics[width=38mm]{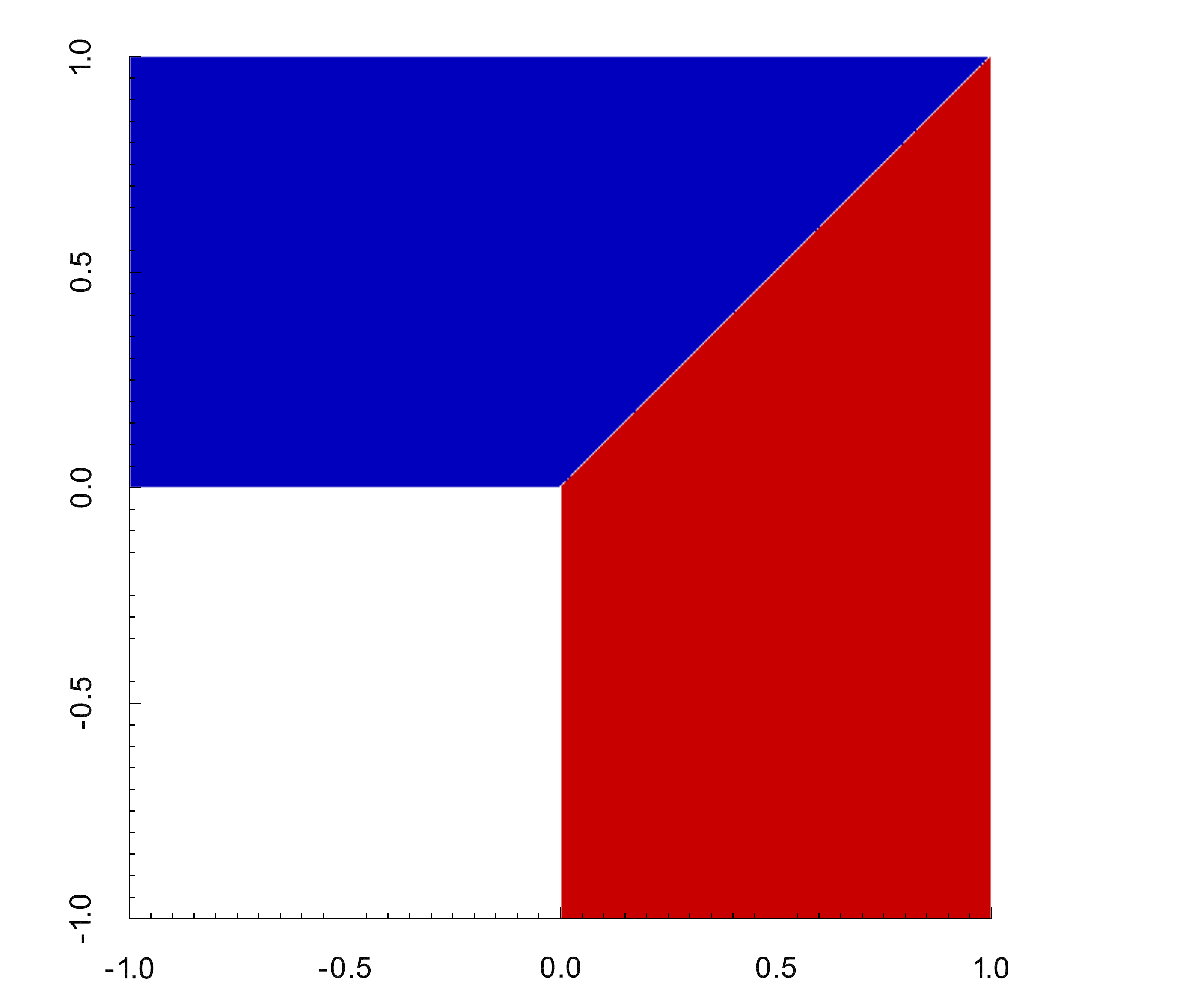}
\includegraphics[width=38mm]{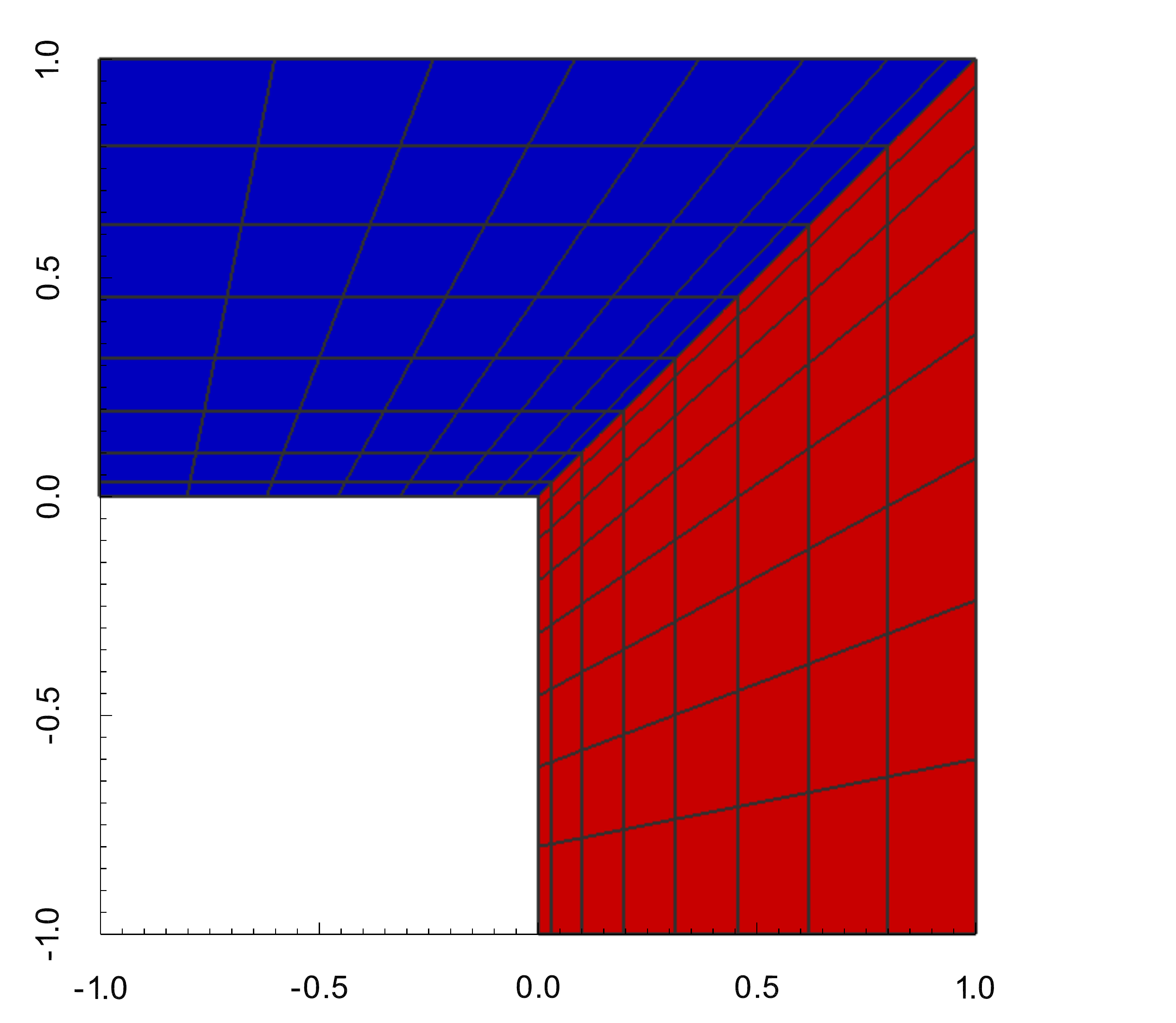}
\includegraphics[width=38mm]{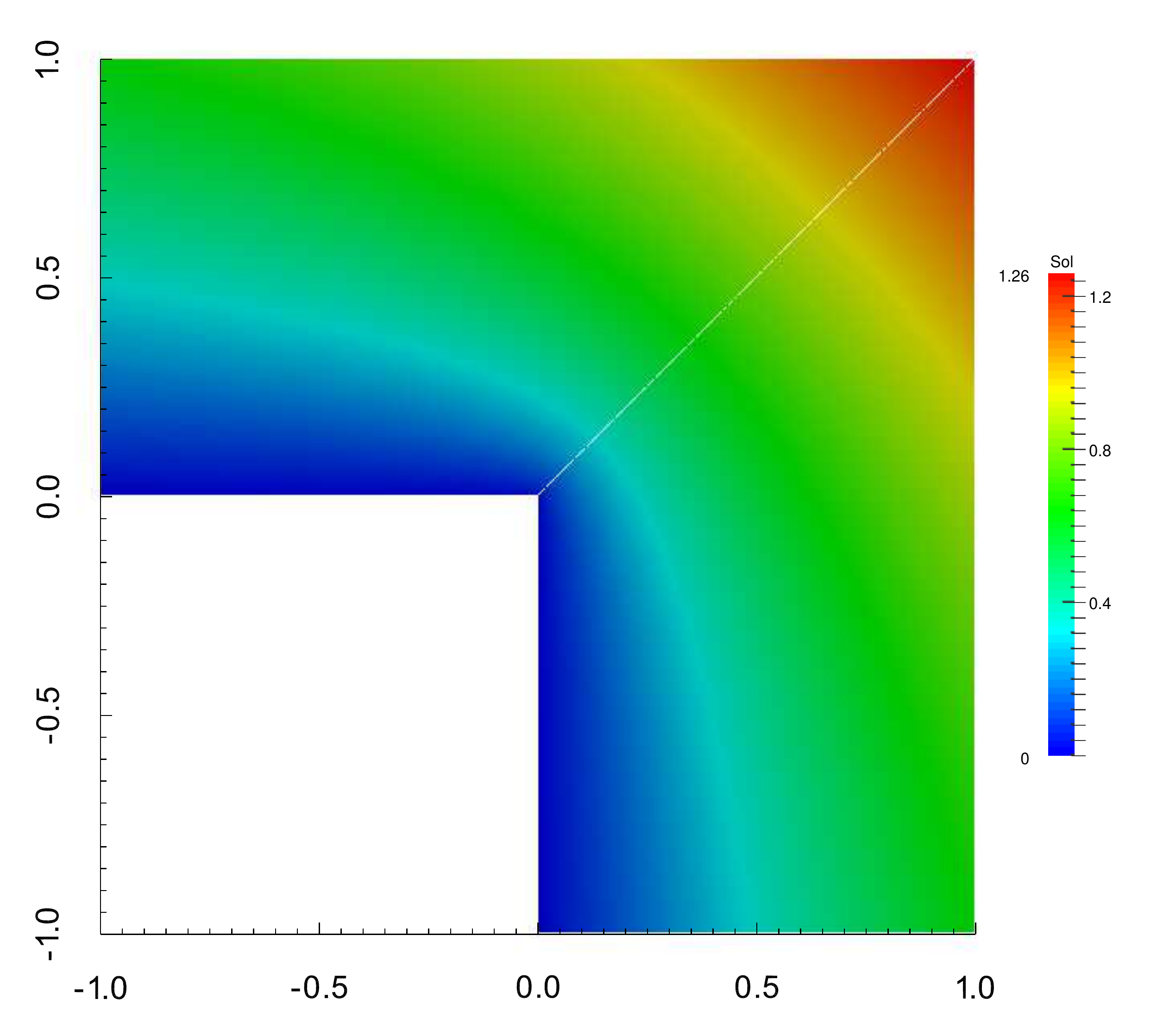}
\caption{L-shape test. Left: subdomains, middle: graded mesh with $\mu=0.6$, right: contours of $u_h$.}
\label{LMMT_fig:2}
\end{figure}

\item \textbf{Interior Point Singularity}\\
The domain is $\Omega=(-1,1)^{2}$.
We consider a solution ${u}$ of (\ref{LMMT_sec:1:eqn:Model1}) with a point singularity at the origin  given by
$u(x)=|x|^{\lambda}$. We set $\lambda=0.6$ and is easy to show that  $u\in W^{l=1.6,2}(\Omega)$.
We set $\alpha=1$ in $\Omega$.
In the left columns of  Table \ref{LMMT_table_GM_PntSingl}, we display the convergence rates for degrees 
$k=1$ and $k=2$ without mesh grading. 
The convergence rates are suboptimal and follow the approximation estimate (\ref{LMMT_5.24}). 
The problem has been  solved again on  graded  
meshes with $\mu=0.6$ for $k=1$ and $\mu=0.3$ for $k=2$, see Fig. \ref{LMMT_fig3}. 
We display the convergence rates in the right columns of Table \ref{LMMT_table_GM_PntSingl}. The rates tend to be optimal as it was expected. 
\end{enumerate}

\begin{figure}
  \centering
  \includegraphics[width=0.32\textwidth]{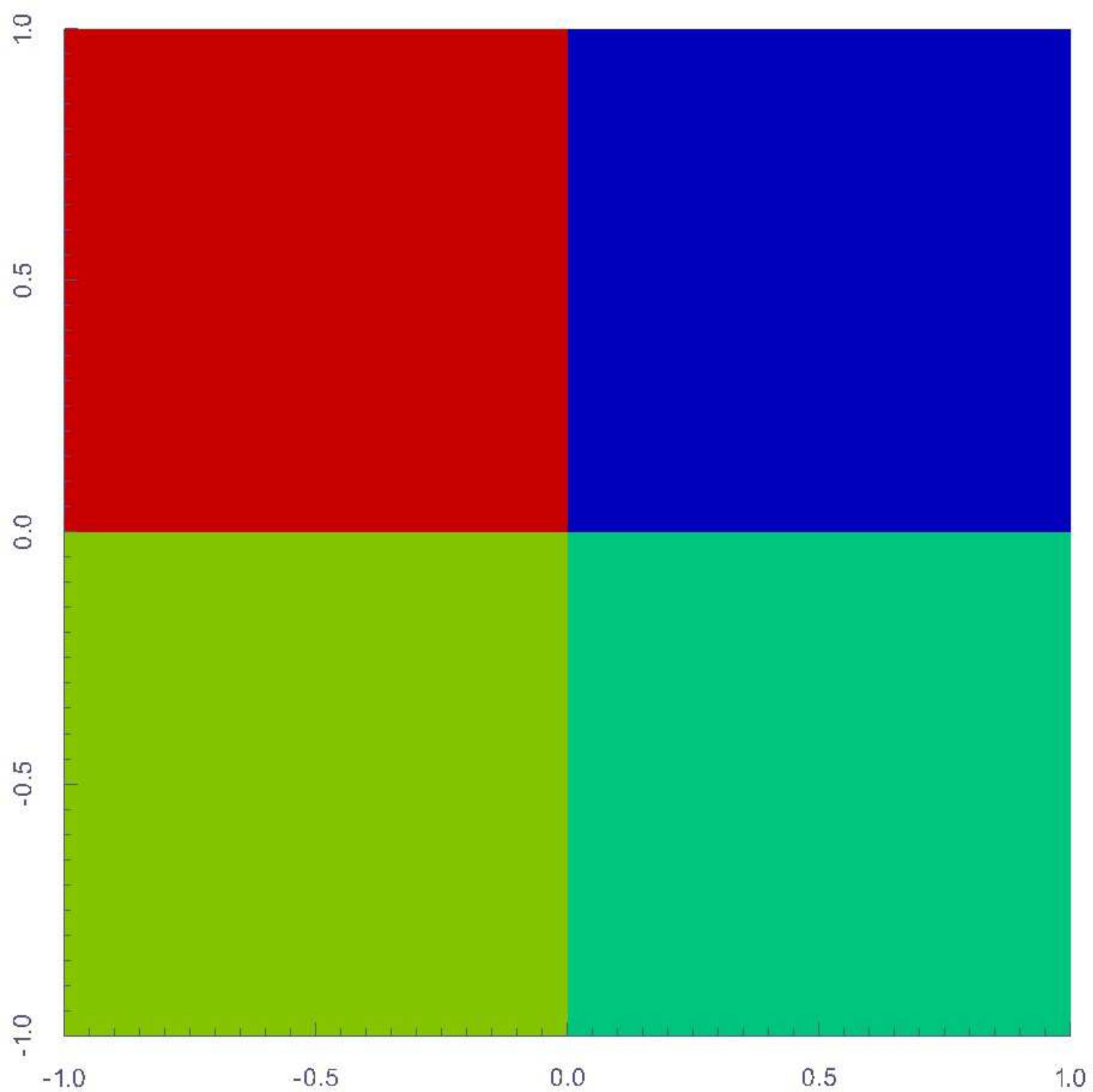}
  \includegraphics[width=0.32\textwidth]{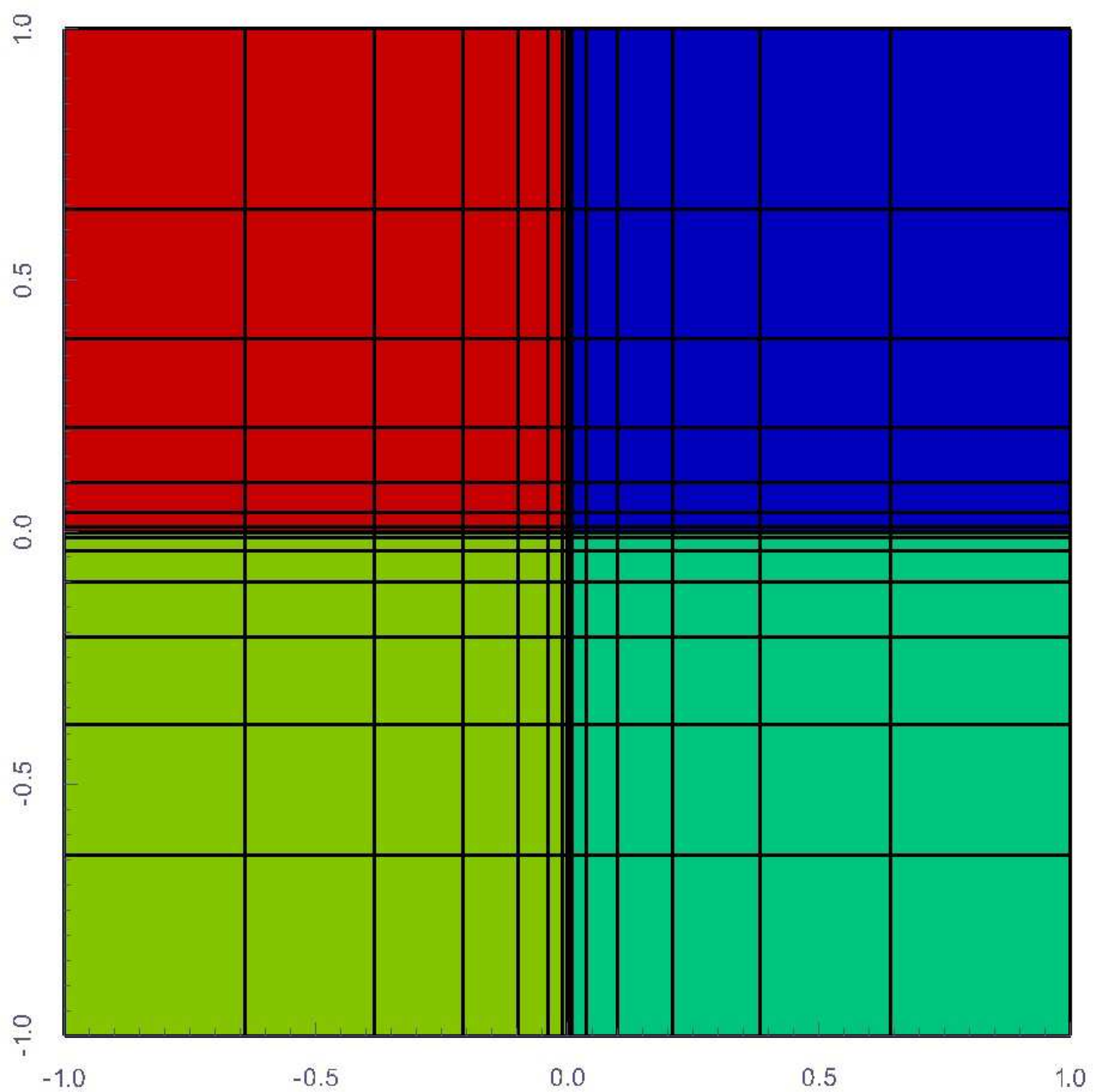}
  \includegraphics[width=0.34\textwidth, height=.196\textheight]{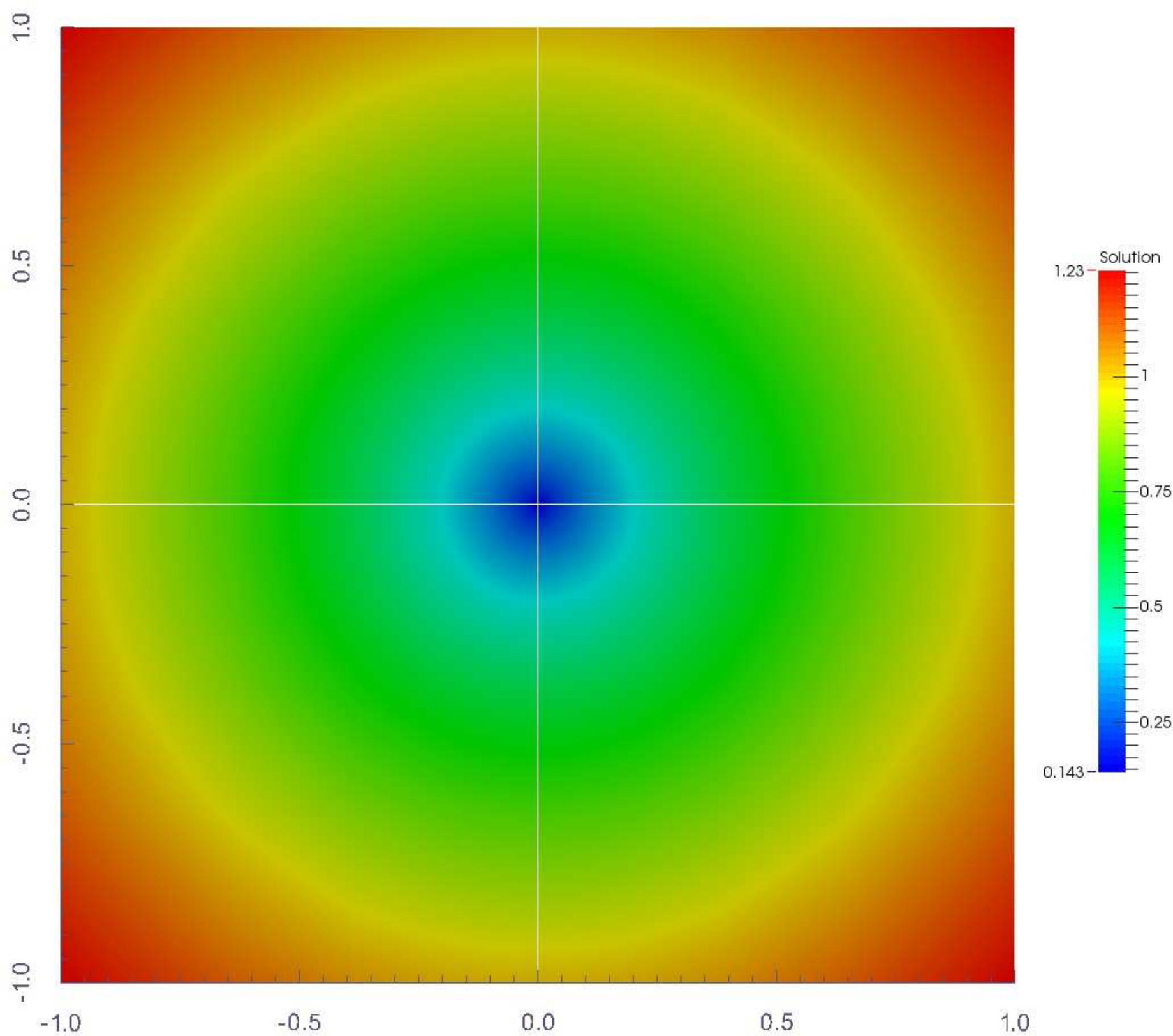}
  \caption{Interior singularity test. 
   Left: the subdomains $\Omega_i,{\ }i=1,...,4$, 
   middle: the graded mesh with $\mu=0.3$, right: the contours of $u_h$.}
    \label{LMMT_fig3}
\end{figure}

\begin{table}[th]
\centering %
\caption{The convergence rates for the  L-shape test (left) and for the interior singularity test (right).}
\label{LMMT_table_GM_PntSingl}
\label{LMMT_ConvRate_Lshape}
\begin{minipage}{5cm}
\begin{tabular}{|c|l|l|l|l|} 
\hline 
          &\multicolumn{2}{|l|}{no grading }
          &\multicolumn{2}{|l|}{with grading } \\ [0.5ex] \hline 
$\frac{h}{2^s}$    &$k=1$& $k=2$& $\begin{matrix}
                                    k=1,\\ 
                                   \mu=0.6
                                    \end{matrix}$  & $\begin{matrix}
                                                        k=2,{\ }\\
                                                        \mu=0.3
                                                     \end{matrix}{\ }$\quad
                                                     \\ \hline 
\multicolumn{5}{|c|} {Convergence rates } \\ [0.5ex] \hline
 $s=0$  &-     & -      & -      &-               \\ 
 $s=1$  &0.636 & 0.650  &0.915   &1.672    {\ }   \\ 
 $s=2$  &0.641 & 0.657  &0.933   &1.919    {\ }   \\ 
 $s=3$  &0.647 & 0.661  &0.946   &1.987    {\ }   \\ 
 $s=4$  &0.652 & 0.666  &0.957   &2.071   {\ }   \\ \hline
\end{tabular}
\end{minipage}
\begin{minipage}{5cm}
\begin{tabular}{|c|l|l|l|l|} 
\hline 
          &\multicolumn{2}{|l|}{no grading }
          &\multicolumn{2}{|l|}{with grading } \\ [0.5ex] \hline 
$\frac{h}{2^s}$    &$k=1$& $k=2$& $\begin{matrix}
                                    k=1,\\ \mu=0.6
                                    \end{matrix}$  
                                & $\begin{matrix}
                                  k=2,{\ } \\ \mu=0.3
                                  \end{matrix}{\ }$\quad \\ \hline 
\multicolumn{5}{|c|} {Convergence rates }        \\ [0.5ex] \hline
 $s=0$  &-     & -      & -      &-              \\ 
 $s=1$  &0.508 &0.573   & 0.855  & 1.661  {\ }   \\ 
 $s=2$  &0.547 &0.580   & 0.930  & 1.828  {\ }   \\ 
 $s=3$  &0.570 &0.586   & 0.969  & 1.928  {\ }   \\ 
 $s=4$  &0.583 &0.591   & 0.990  & 1.951  {\ }   \\ \hline
\end{tabular}
\end{minipage}
\end{table}



\section{Multipatch dG IgA for PDEs on Surfaces}
\label{LMMT_sec:3:Surfaces}

\subsection{Diffusion Problems on Open and Closed Surfaces}
\label{LMMT_subsec:3.1:DiffusionProblemsonOpenandClosedSurfaces}

Let us now consider a  diffusion problem of the form (\ref{LMMT_sec:1:eqn:Model1})
on a sufficiently smooth, open surface $\Omega$, the weak formulation of which 
can formally be written in the same form as (\ref{LMMT_sec:2:eqn:Model1-VF}) 
in Section~\ref{LMMT_sec:2:Volumetric}:
find  a function $u\in W^{1,2}(\Omega)$ such that $u=u_D$ 
on the boundary $\partial \Omega$ of the surface $\Omega$ 
and  satisfies the variational formulation 
%
%
\begin{equation}
\label{LMMT_sec:3:eqn:SurfaceDiffusion-VF}
 a(u,v)=l(v), \; \forall  v \in W^{1,2}_{0}(\Omega),
\end{equation}
with the bilinear and linear forms 
$a(\cdot,\cdot)$ and  $l(\cdot)$, 
but now defined by the relations
\begin{equation*}
	a(u,v) = \int_\Omega \alpha \, \nabla_\Omega u \cdot \nabla_\Omega v \, d\Omega
	 \quad \mbox{and} \quad
	l(v)  = \int_\Omega f v \, d\Omega,
\end{equation*}
respectively, where $\nabla_\Omega$ denotes the 
surface
gradient, see, e.g., Definition~2.3 in \cite{LMMT_DziukElliott:2013a}
for its precise description. 
For simplicity of the presentation,
we here assume Dirichlet boundary condition.
It is clear that other boundary conditions can be treated 
in the same framwork as it was done in \cite{LMMT_LangerMoore:2014a} 
for mixed boundary conditions. In the case of open surefaces with pure Neumann boundary condition 
and closed surfaces, we look for a solution $u\in W^{1,2}(\Omega)$ satifying  
the uniqueness condition $\int_\Omega u(x) dx = 0$ and the 
variational equation (\ref{LMMT_sec:2:eqn:Model1-VF}) 
under the solvability condition $l(1) = 0$. 
In  Subsection~\ref{LMMT_subsec:3.4:NumericalExamples}, 
we present and discuss the numerical results obtained 
for different diffusion problems on an open (Car) and on two 
closed  (Sphere, Torus) surfaces.

\subsection{Multipatch dG IgA Discretization}
\label{LMMT_subsec:3.2:MultipatchdGIgADiscretization}

Let $\mathcal{T}_{H}(\Omega) = \{\Omega_{i}\}_{i=1}^{N}$ be again a partition of 
our physical computational domain $\Omega$, that is now a surface, 
into non-overlapping patches (sub-domains) $\Omega_{i}$ such that 
(\ref{LMMT_1}) holds,
and let each patch $\Omega_{i}$ be the image of the 
parameter domain  $\widehat{\Omega} = (0,1)^2 \subset \mathbb{R}^{2}$ 
by some NURBS mapping 
$\Phi_i: \widehat{\Omega} \rightarrow \Omega_{i} \subset \mathbb{R}^{3},\; 
\hat{x} = (\hat{x}_1,\hat{x}_2) 
\mapsto 
{x} = ({x}_1,{x}_2,{x}_3)=\Phi_i(\hat{x})$, 
which can be represented in the form
\begin{equation}
\label{LMMT_GeometricalMappingRepresentation}
  \Phi_i(\hat{x}_{1},\hat{x}_{2}) = \sum_{k_{1}=1}^{n_{1}} \sum_{k_{2}=1}^{n_{2}}
 {C}^{(i)}_{(k_{1},k_{2})} \widehat{B}^{(i)}_{(k_{1},k_{2})}(\hat{x}_{1},\hat{x}_{2})
\end{equation}
where $\{ \hat{B}^{(i)}_{(k_{1},k_{2})} \}$ are the bivariate NURBS basis functions, 
and $\{{C}^{(i)}_{(k_{1},k_{2})} \}$ are the control points, 
see \cite{LMMT_CottrellHughesBazilevs:2009a} for a detailed 
description. We always assume that the mapping $\Phi_i$ is regular.
Therefore, the inverse mapping $\hat{x} = {\Psi}_i(x):=\Phi^{-1}_i(x)$
is well defined for all patches $\Omega_{i}$, $i=1,\ldots,N$.

Now the dG IgA scheme for solving our surface diffusion problem 
(\ref{LMMT_sec:3:eqn:SurfaceDiffusion-VF}) can formally be written 
in the form (\ref{LMMT_8a}) as in Section~\ref{LMMT_sec:2:Volumetric}:
find $u_h\in \mathbb{B}_h(\mathcal{T}_H(\Omega))$ such that
\begin{equation}\label{LMMT_:sec:3:eqn:dgIgAScheme-SurfaceDiffusion}
 a_h(u_h,v_h)= l(v_h) + p_D(u_D,v_h), \; \forall v_h \in \mathbb{B}_h(\mathcal{T}_H(\Omega)),
\end{equation}
where $a_h(\cdot,\cdot)$ and $p_D(\cdot,\cdot)$ are defined by (\ref{LMMT_8b})
provided that we replace the gradient $\nabla$ by the surface gradient $\nabla_\Omega$.
Since the dG bilinear form $a_h(\cdot,\cdot)$ is positive on 
$\mathbb{B}_h(\mathcal{T}_H(\Omega)) \setminus \{0\}$ for sufficiently 
large $\mu$, cf. Lemma~\ref{LMMT_lemma5}, there exist a unique dG solution $u_h\in \mathbb{B}_h(\mathcal{T}_H(\Omega))$. 
The dG IgA scheme (\ref{LMMT_:sec:3:eqn:dgIgAScheme-SurfaceDiffusion}) is equivalent 
to a system of algebraic equations of the form 
\begin{equation}\label{LMMT_:sec:3:eqn:dgIgA-System}
 K_h \underline{u}_h = \underline{f}_h,
\end{equation}
the solution $\underline{u}_h$ of which gives us the coefficients (control points) 
of $u_h$. In order to generate the entries of the system matrix $K_h$ and
the right-hand side $\underline{f}_h$, we map the patches $\Omega_i$ composing the 
physical domain, i.e., our surface $\Omega$, into the parameter domain $\widehat{\Omega} = (0,1)^2$.
For instance, for the broken part $a_i(u_h,v_h)$ of the bilinear form $a_h(u_h,v_h)$, 
we obtain 
\begin{align*}
 a_i(u_h,v_h) &= \int_{\Omega_i} \alpha^{(i)}\, \nabla_\Omega u_h(x) \cdot \nabla_\Omega u_h(x)\, d\Omega \\
	      &= \int_{\widehat{\Omega}_i} \alpha^{(i)} \, 
		  [J_i(\hat{x}) F^{-1}_i(\hat{x})  \hat{\nabla} \hat{u}_i(\hat{x}) ]^{\top} 
		  [J_i(\hat{x}) F^{-1}_i(\hat{x})  \hat{\nabla} \hat{v}_i(\hat{x}) ]  g_i(\hat{x})\, d\hat{x}\\
	      &= \int_{\widehat{\Omega}_i} \alpha^{(i)} \, (\hat{\nabla}\hat{u}_i(\hat{x})) ^{\top}
		  \, F^{-1}_i(\hat{x}) \, 
		  \hat{\nabla} \hat{v}_i(\hat{x})\, g_i(\hat{x}) d\hat{x} \, ,
\end{align*}
where $J_i(\hat{x}) = \partial \Phi_i(\hat{x}) / \partial \hat{x}$,
      $F_i(\hat{x})= (J_i(\hat{x}))^\top (J_i(\hat{x}))$ and
      $ g_i(\hat{x}) = (\det F_i(\hat{x}))^{1/2}$
denote the Jacobian, the first fundamental form and the square root
of the associated determinant, respectively.
These terms, coming from the parameterization of the domain, can be
exploited for deriving efficient matrix assembly methods,
cf. \cite{LMMT_iil2014, LMMT_iil2014mc}.
Furthermore, we use the notations $\hat{u}_i (\hat{x})= u_h(\Phi_i(\hat{x}))$
and $\hat{\nabla} = (\frac{\partial}{\partial \hat{x}_1},\frac{\partial}{\partial \hat{x}_2})^{\top} $.

\subsection{Discretization Error Estimates}
\label{LMMT_subsec:3.3:DiscretizationErrorEstimates}

In \cite{LMMT_LangerMoore:2014a}, we derived discretization error estimates 
of the form 
\begin{equation}
\label{LMMT_:sec:3:eqn:DGnormErrorEstimate1}
\|u-u_{h} \|_{dG}^2 \leq C 
\sum_{i=1}^{N} \alpha^{(i)} h_{i}^{2t}\|u\|^{2}_{H^{1+t}(\Omega_{i})},
\end{equation}
with $t:= \min\{s,k\} $, provided that the solution $u$ of our surface diffusion 
problem (\ref{LMMT_sec:3:eqn:SurfaceDiffusion-VF}) 
belongs to $H^{1+s}(\mathcal{T}_{H}(\Omega)) = W^{1+s,2}(\mathcal{T}_{H}(\Omega))$ with some $s > 1/2$. 
In the case $t=k$, estimate (\ref{LMMT_:sec:3:eqn:DGnormErrorEstimate1}) 
yields the convergence rate $\mathcal{O}(h^k)$ with respect to  the dG norm,
whereas the Aubin-Nitsche trick provides the faster rate $\mathcal{O}(h^{k+1})$ 
in the $L_2$ norm. Here $k$ always denotes the underlying polynomial degree of the NURBS.
This convergence behavior is nicely confirmed by our numerical experiments 
presented in \cite{LMMT_LangerMoore:2014a} and in the next subsection.

In \cite{LMMT_LangerMoore:2014a}, we assumed matching meshes and 
some regularity of the solution of (\ref{LMMT_sec:3:eqn:SurfaceDiffusion-VF}), 
namely $u \in H^{1+s}(\mathcal{T}_{H}(\Omega))$.
It is clear that the results of Theorem~\ref{LMMT_:sec:2:th1:dGnormErrorEstimate},
which includes  no-matching meshes and low-regularity solutions,
can easily be carried over to diffusion problems on open and closed surfaces.
The same is true for mesh grading techniques presented in 
Subsection~\ref{LMMT_subsec:2.6:GradedMeshPartitionsforthedG IgAMethods}.

\subsection{Numerical Examples}
\label{LMMT_subsec:3.4:NumericalExamples}


\subsubsection{Sphere}
\label{LMMT_subsubsec:3.4.1:Sphere}


Let us start with a diffusion problem on a closed surface $\Omega$
that is given by the sphere 
$\Omega = \{(x,y,z) \in (-1,1)^3:\;  x^2+y^2+z^2 = 1\}$
with 
unitary radius.
The computational domain $\Omega$ is decomposed into 6 patches, 
see left-hand side of Fig.~\ref{LMMT_sec:3:fig:sphere}.
The knot vectors representing the geometry of each patch are given as 
$\Xi_{1,2}= (0, 0, 0, 0, 0, 1, 1, 1, 1, 1)$
in both directions.  
Since the surface is closed, 
we impose the uniqueness constraint $\int_{\Omega} u\, d\Omega = 0$ on the solution. 
The right-hand side 
$f(r,\phi, \theta) = 12u(r,\phi, \theta)$, 
where 
the solution
$u(r,\phi, \theta) = 12\sin(3\phi)\sin^3(\theta)$ is an eigenfunction of the Laplace-Beltrami
operator $(-\Delta_{\Omega})$ satisfying the compatibility 
condition $\int_{\Omega} f\, d\Omega = 0.$ 
The example can also be found in \cite{LMMT_GrossReusken:2011a}.
 \begin{figure}[th!]
  \begin{center}
      \includegraphics[width=0.51\textwidth]{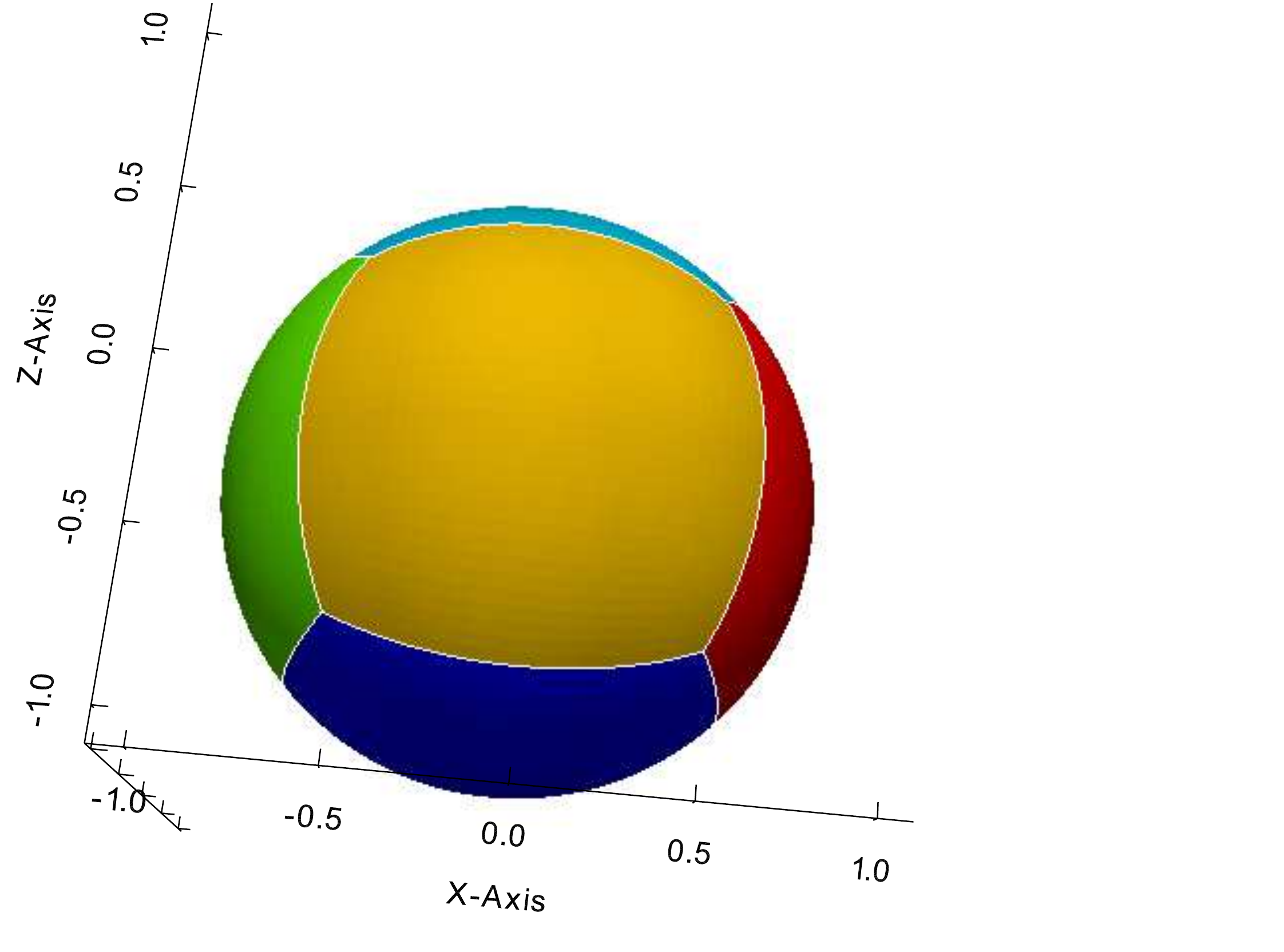}
      \includegraphics[width=0.48\textwidth]{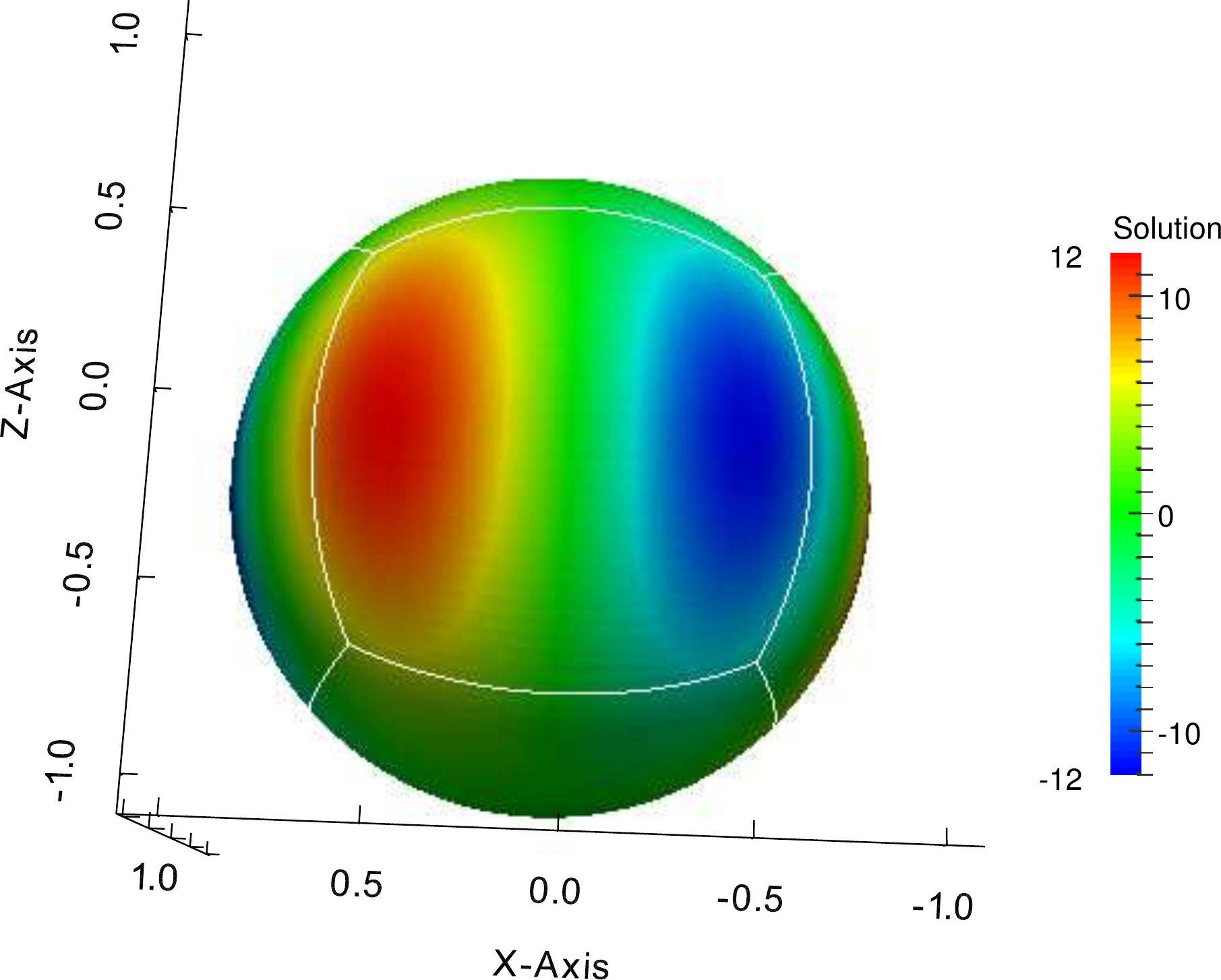}
        \caption{Sphere: Geometry and  decomposition into 6 patches (left) and 
		  Solution for the Laplace-Beltrami problem (right).
	  }
    \label{LMMT_sec:3:fig:sphere}
  \end{center}
\end{figure} 
The diagrams displayed in Fig.~\ref{LMMT_sec:3:fig:sphere:errors}
show the error decay with respect to the $L_2$ (left) and dG (right) norms 
for polynomial degrees $k = 1, 2, 3,$ and $4$. 
As expected by our theoretical results, we observe the full convergence rates 
$\mathcal{O}(h^{k+1})$ and $\mathcal{O}(h^{k})$ for the $L_2$ norm 
and the dG norm, respectively. 
\begin{figure}[th!]
  \centering
  \includegraphics[width=0.5\textwidth, height = 0.20\textheight]{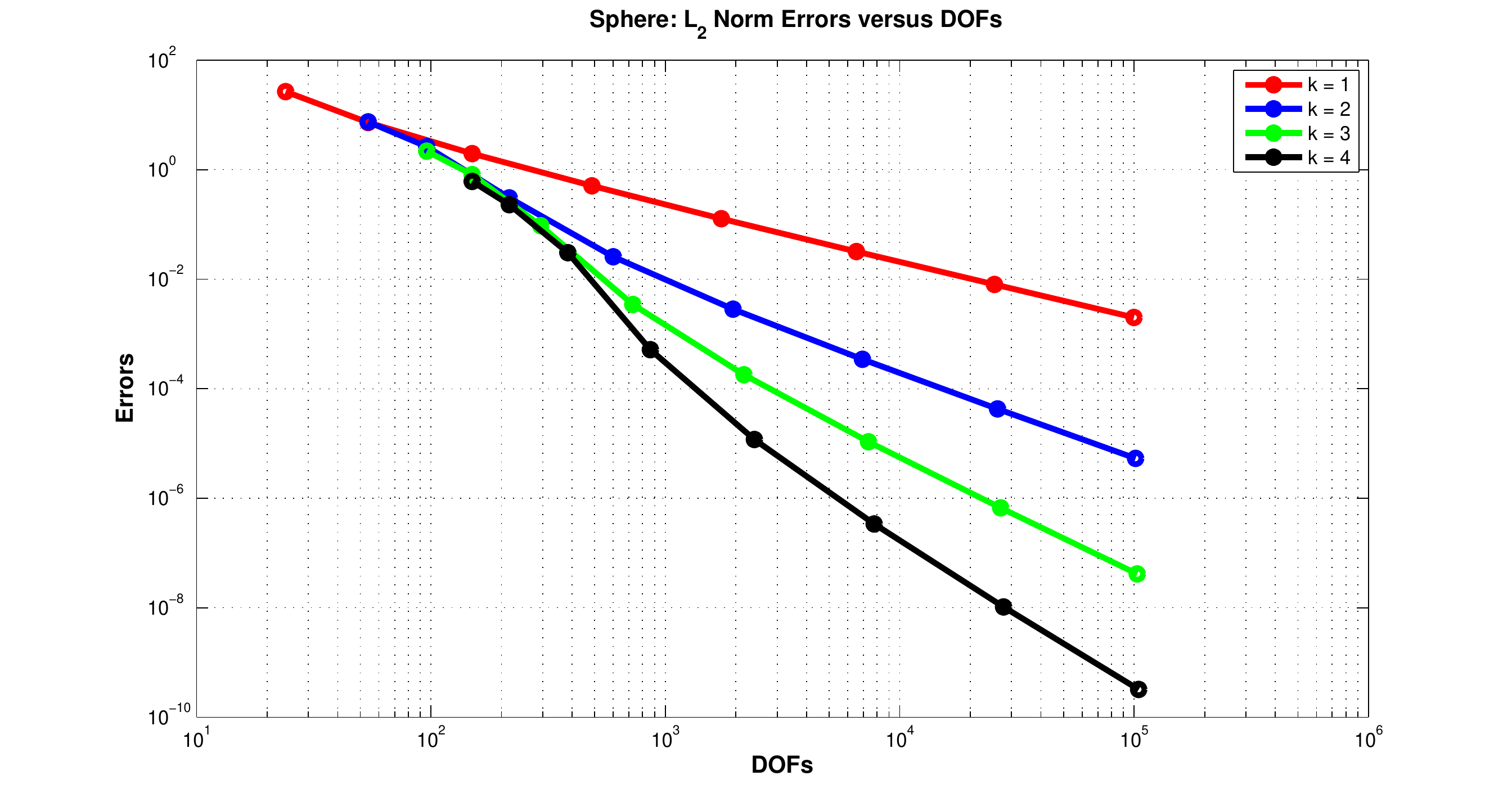}
  \includegraphics[width=0.5\textwidth, height = 0.20\textheight]{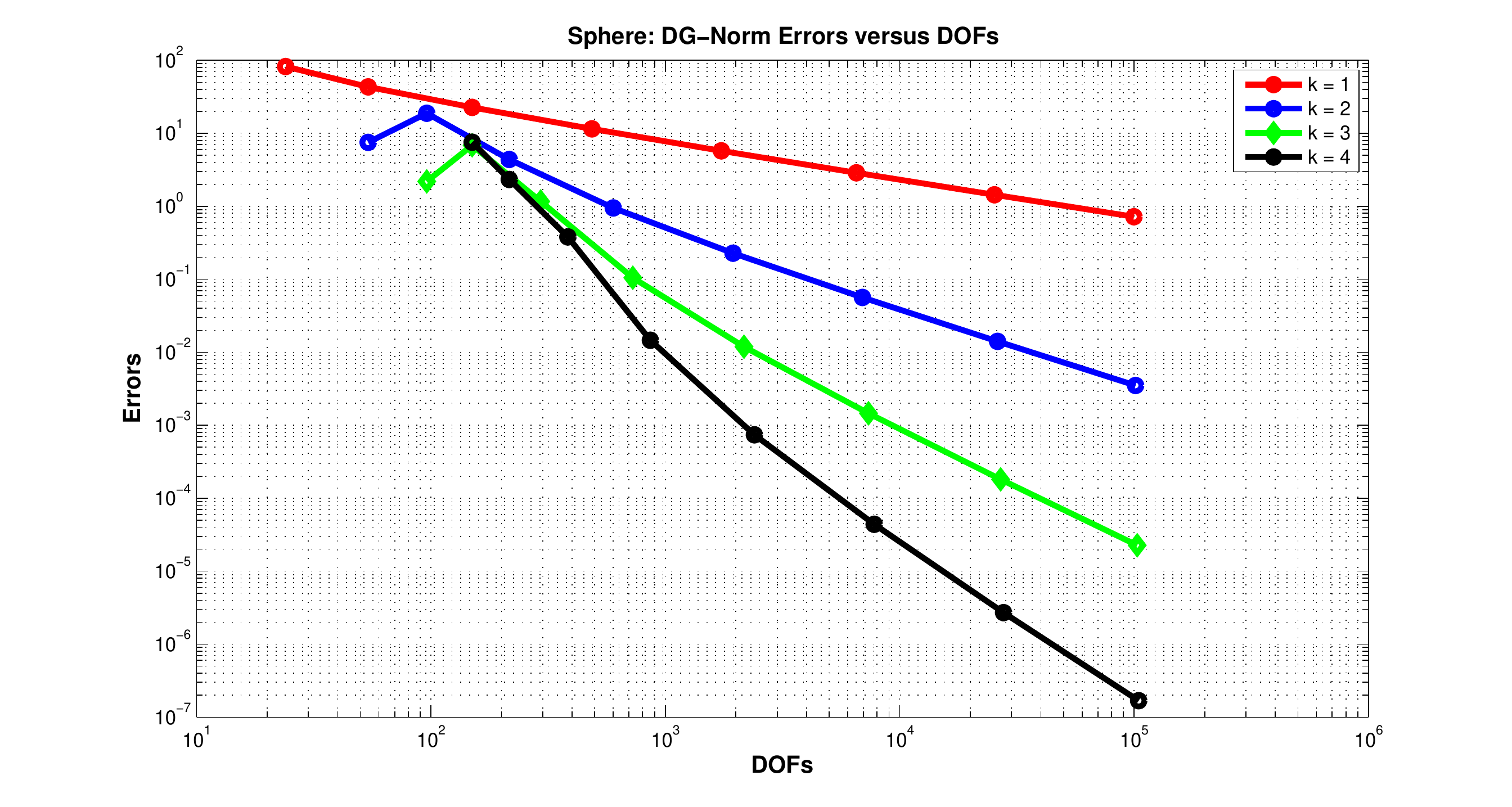}
  \caption{Sphere: Error decay in the $L_2$ (left) and dG (right) norms for polynomial 
		  degrees 1 to 4. 
	  }
  \label{LMMT_sec:3:fig:sphere:errors}
\end{figure}
In Table~\ref{LMMT_sec:3:tab:sphere:cGdGL2}, we compare the $L_2$ errors of dG IgA solutions
with those produced by the corresponding continuous (standard) Galerkin (cG) IgA scheme 
for the polynomial degree $k=5$. In the case of smooth solutions, 
the dG IgA is as good as the cG counterpart. The same is true for the errors with respect 
to the dG norm.

\begin{table}[th!]
 \centering
 \caption{Sphere: Comparison of the cG and dG IgA error decay in the $L_2$ norm for $k=5$.}
  \begin{tabular}{|l|l|l|l|l|l|l|l|}
   \hline
 $k=5$   &\multicolumn{2}{|c|}{cG-IgA} &\multicolumn{2}{|c|}{dG-IgA} \\
     \cline{1-5}   
       Dofs    &$L_2$ error     &conv. rate &$L_2$ error     &conv. rate    \\  \hline
        216    &0.168803       &0           &0.166582        &0		\\
        294    &0.0602254      &1.48689     &0.0599854       &1.16908		\\
        486    &0.00900833     &2.74104     &0.00898498      &2.55496		\\
       1014    &8.90909$\times 10^{-5}$     &6.65984      &8.90774$\times 10^{-5}$   &6.537		\\
       2646    &8.64021$\times 10^{-7}$     &6.68807      &8.63906$\times 10^{-7}$   &6.85905		\\
       8214    &1.16592$\times 10^{-8}$     &6.21152      &1.16582$\times 10^{-8}$   &6.27626		\\
      28566    &1.75119$\times 10^{-10}$    &6.05699      &1.75127$\times 10^{-10}$ &6.06894		\\
 \hline
 \end{tabular}
\label{LMMT_sec:3:tab:sphere:cGdGL2}
\end{table}



\subsubsection{Torus}
\label{LMMT_subsubsec:3.4.2:Torus}


We now consider the closed surface 
$$\Omega = \{(x,y) \in (-3,3)^2, z \in (-1,1): \; r^2 = z^2 + (\sqrt{x^2+y^2} -R^2) \}$$ 
that is nothing but a torus that is decomposed into 4 patches, 
see Fig.~\ref{LMMT_sec:3:fig:torus} left.
The knot vectors
$\Xi_{1}  = \{0, 0, 0, 0.25, 0.25, 0.50, 0.50, 0.75, 0.75, 1, 1, 1\}$
and
$\Xi_{2}= \{0, 0, 0, 1, 1, 1\}$
describe the NURBS used for the geometrical representation of the patches.
We first consider  the surface Poisson equation,
also called  Laplace-Beltrami equation,
with the right-hand side 
\begin{align*}
 f(\phi, \theta) &= r^{-2} \left[9\sin(3\phi)\cos(3\theta+\phi)\right] \\
&- \left[(R + r\cos(\theta))^{-2} (-10\sin(3\phi)\cos(3\theta+\phi) - 6\cos(3\phi)\sin(3\theta+\phi)) \right]\\
&- \left[ (r(R+r\cos(\theta))^{-1}) (3\sin(\theta)\sin(3\phi)\sin(3\theta+\phi)) \right],
\end{align*}
where $\phi = \arctan (y/x)$, 
$\theta = \arctan (z/(\sqrt{x^2+y^2}-R))$ , $R = 2$,
and $r=1$.
The exact solution is given by $u= \sin(3 \phi)\cos(3\theta + \phi)$,
cf. also \cite{LMMT_GrossReusken:2011a}.
We mention that the functions $u$ and $f$ are chosen  such that 
the zero mean compatibility condition holds.
The IgA approximation to this solution is depicted in middle picture
of Fig.~\ref{LMMT_sec:3:fig:torus}.
\begin{figure}[th!]
  \centering
  \includegraphics[width=0.33\textwidth]{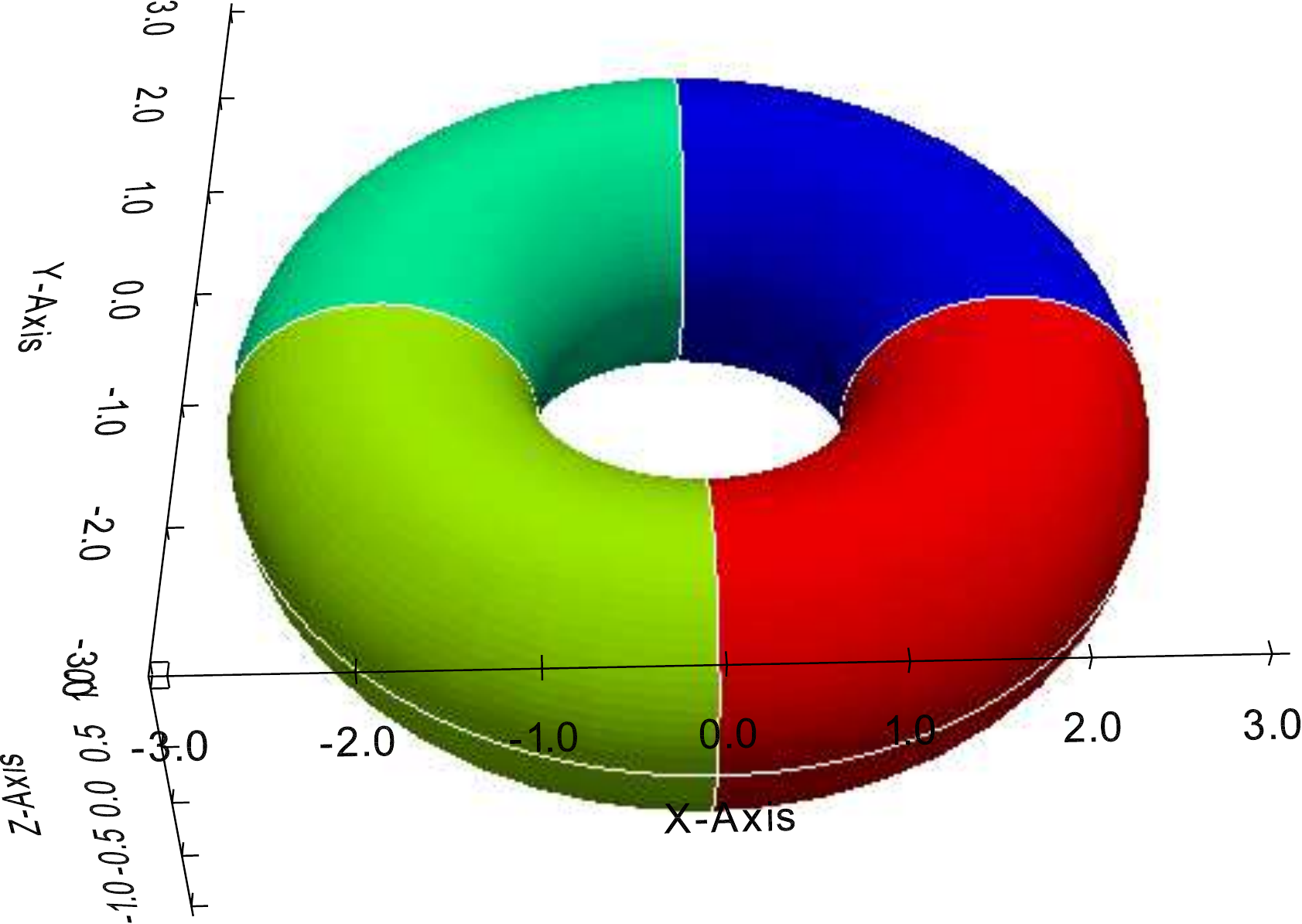}
  \includegraphics[width=0.33\textwidth]{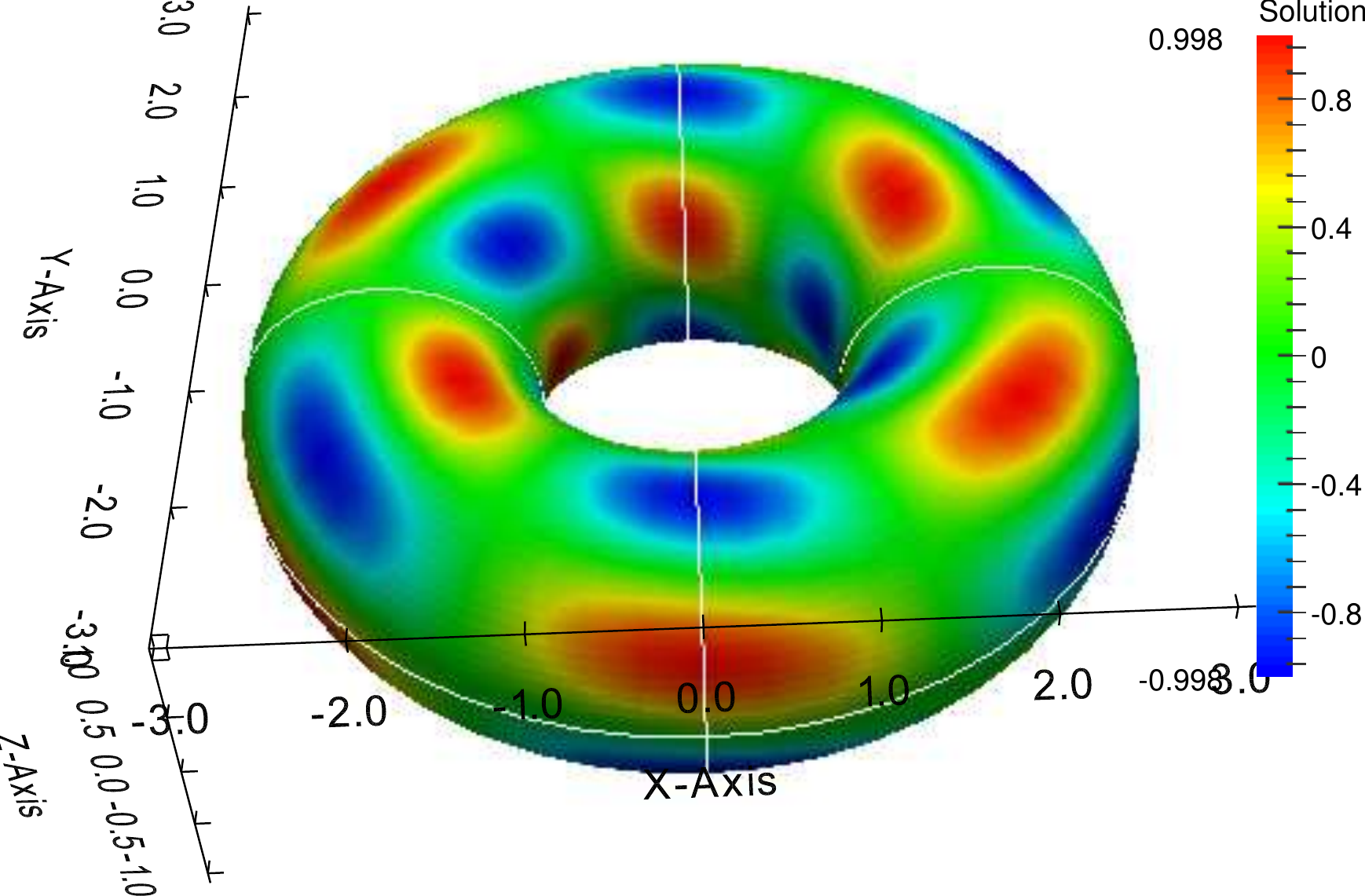}
  \includegraphics[width=0.31\textwidth]{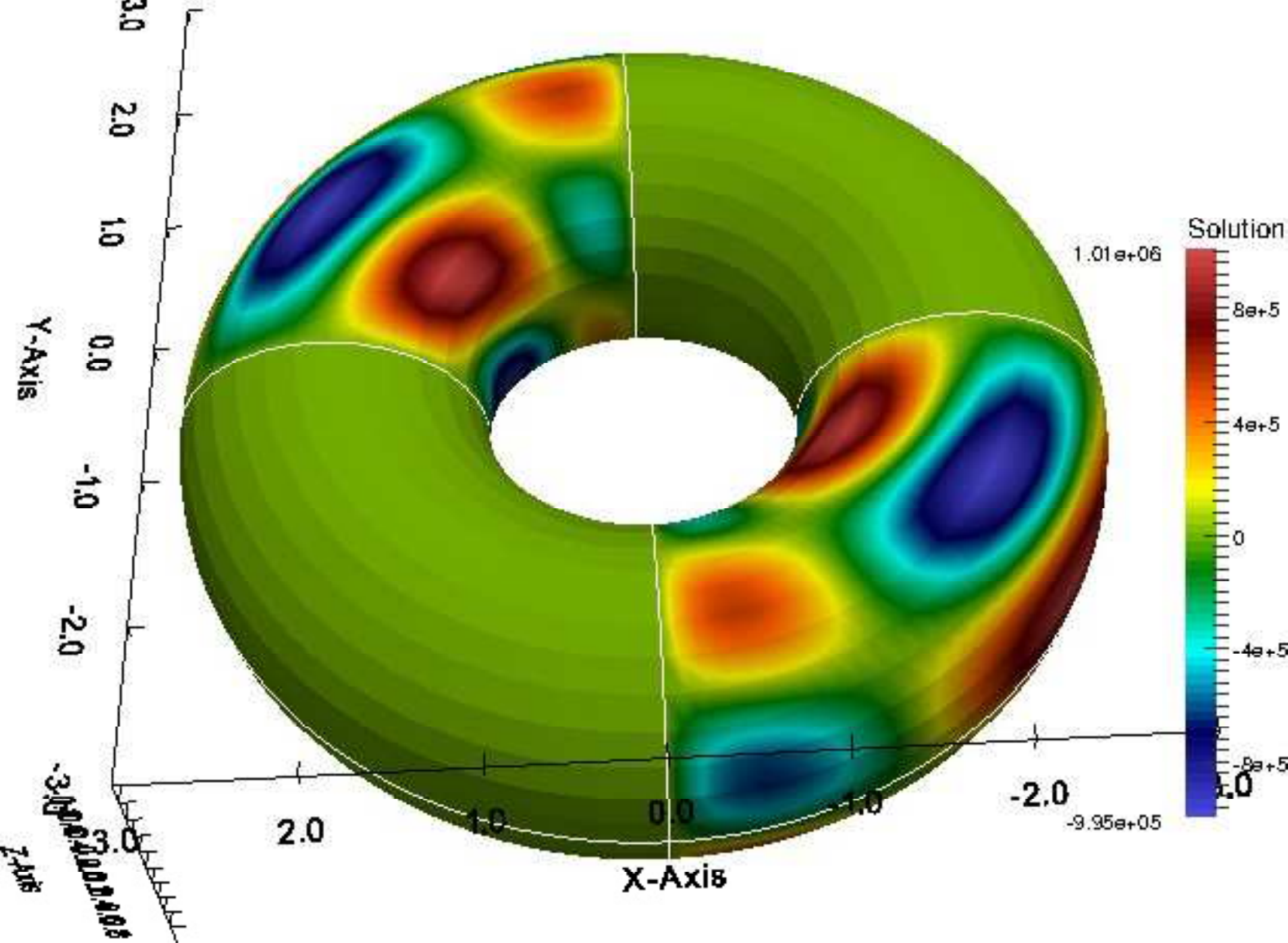}
  \caption{Torus: Geometry and  decomposition into 4 patches (left),
		  Solution for the Laplace-Beltrami problem  (middle), 
		  Solution for jumping coefficents (right).
	  }
  \label{LMMT_sec:3:fig:torus}
\end{figure}
The diagrams displayed in Fig.~\ref{LMMT_sec:3:fig:torus:errors}
show the error decay with respect to the $L_2$ (left) and dG (right) norms 
for polynomial degrees $k = 1, 2, 3,$ and $4$. 
As expected by our theoretical results, we observe the full convergence rates 
$\mathcal{O}(h^{k+1})$ and $\mathcal{O}(h^{k})$ for the $L_2$ norm 
and the dG norm, respectively. 
\begin{figure}[th!]
  \centering
  \includegraphics[width=0.5\textwidth, height = 0.20\textheight]{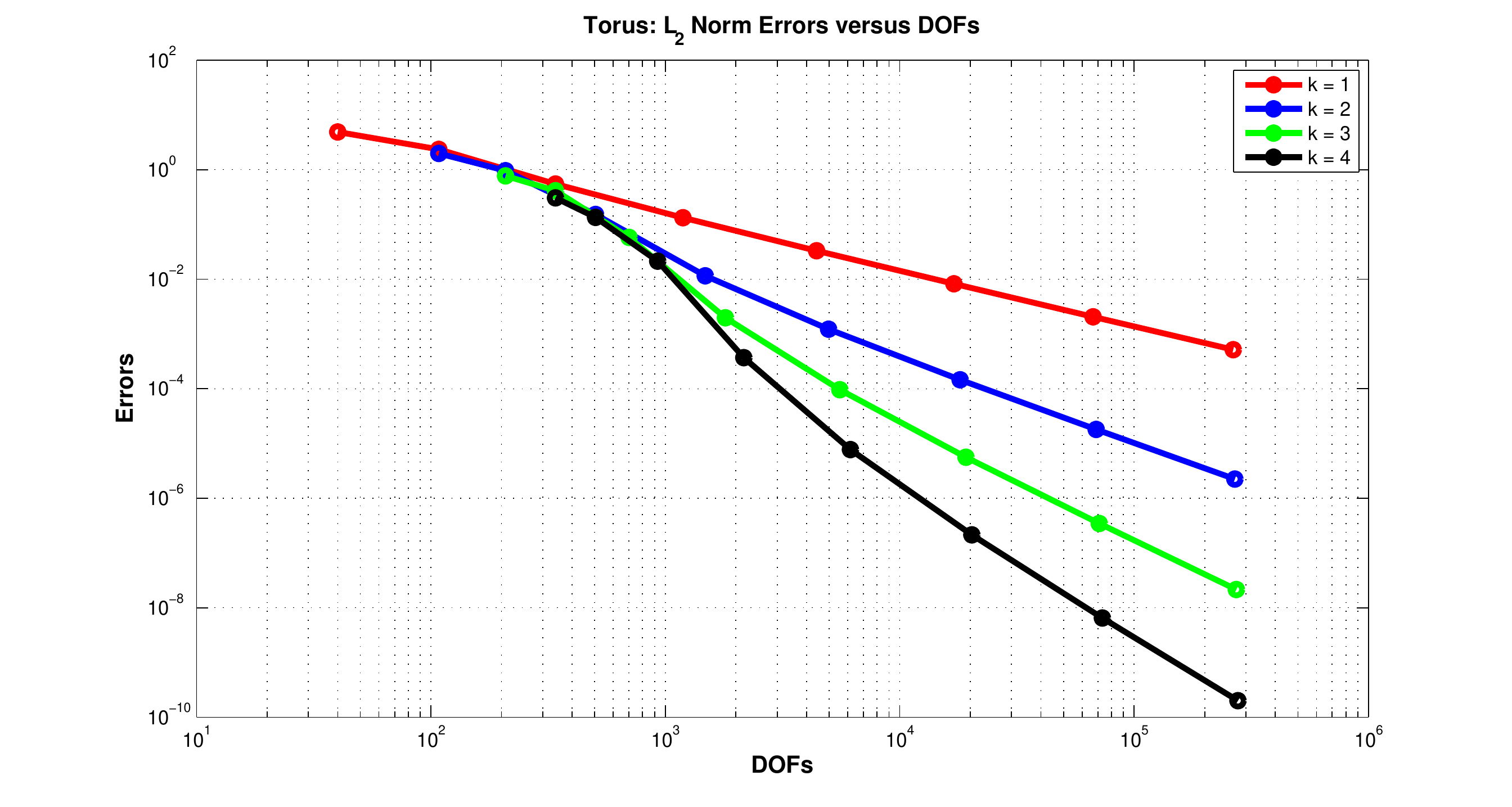}
  \includegraphics[width=0.5\textwidth, height = 0.20\textheight]{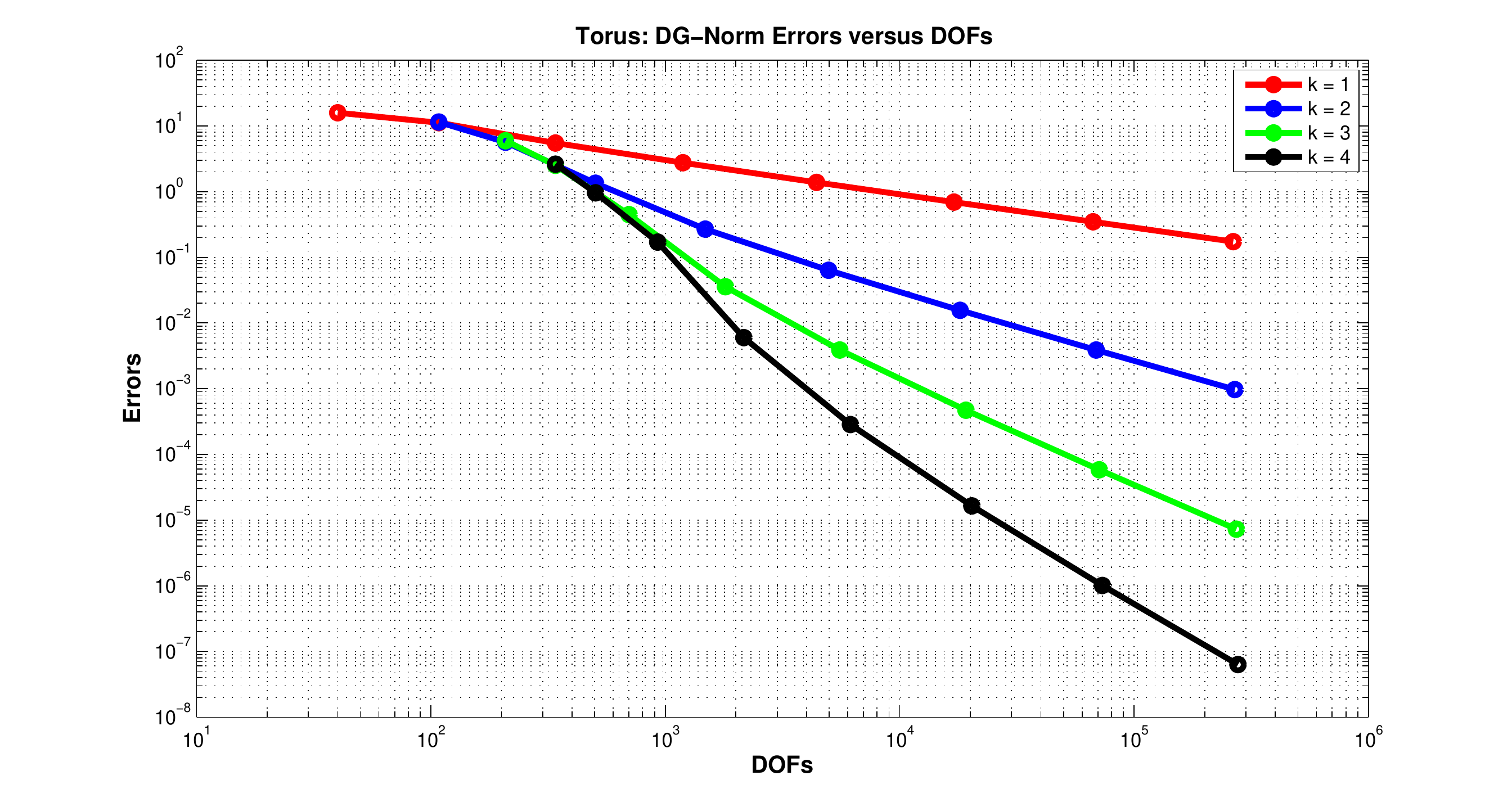}
  \caption{Torus: Error decay in the $L_2$ (left) and dG (right) norms for polynomial 
		  degrees 1 to 4.}
  \label{LMMT_sec:3:fig:torus:errors}
\end{figure}
In Table~\ref{LMMT_sec:3:tab:torus:cGdGL2}, we compare the $L_2$ errors of dG IgA solutions
with those produced by the corresponding 
cG IgA scheme 
for the polynomial degree $k=5$. In the case of smooth solutions, 
the dG IgA is as good as the cG counterpart. The same is true for the errors with respect 
to the dG norm.

%

\begin{table}[th!]
 \centering
 \caption{Torus: Comparison of the cG and dG IgA error decay in the $L_2$ norm for $k=5$.}
  \begin{tabular}{|l|l|l|l|l|l|l|l|}
   \hline
 $k=5$   &\multicolumn{2}{|c|}{cG-IgA} &\multicolumn{2}{|c|}{dG-IgA} \\
     \cline{1-5}   
       Dofs    &$L_2$ error                 &conv. rate   &$L_2$ error                 &conv. rate    \\  \hline
        504    &0.0974255                   &0            &0.0973029                   &0		\\
        700    &0.0433491                   &1.1683       &0.043271                    &1.16908		\\
       1188    &0.00736935                  &2.55639      &0.00736339                  &2.55496		\\
       2548    &7.9296$\times 10^{-5}$      &6.53814      &7.92948$\times 10^{-5}$     &6.537		\\
       6804    &6.83083$\times 10^{-7}$     &6.85904      &6.83068$\times 10^{-7}$     &6.85905		\\
      21460    &8.8131$\times 10^{-9}$      &6.27627      &8.81294$\times 10^{-9}$     &6.27626		\\
      75348    &1.3131$\times 10^{-10}$     &6.0686       &1.31276$\times 10^{-10}$    &6.06894		\\
 \hline
 \end{tabular} 
\label{LMMT_sec:3:tab:torus:cGdGL2}
\end{table}

Now we consider the case when the diffusion coefficients $\alpha$ have large jumps 
across the boundaries of the patches in which the torus was decomposed.
More precisely, we assume that the diffusion coefficient $\alpha = \alpha^{(i)} > 0$ 
in the patch $\Omega_i$, $i=1,2,3,4$, where
$\alpha^{(2)} = \alpha^{(4)} = 1$ and $\alpha^{(1)} = \alpha^{(3)} =10^{-6}$.
The patches are arranged from blue $\Omega_{1}$ to red $\Omega_{4}.$
The right-hand side $f$ is the same as given above for the 
case of the Laplace-Beltrami problem, i.e. for $\alpha = 1.$
Now the solution is not known, but we know that the solution is smooth in the patches 
$\overline{\Omega_i}$
and has steep gradients towards $\partial \Omega_2$ and $\partial \Omega_4$,
see also of Fig.~\ref{LMMT_sec:3:fig:torus} right.
Due to our theory, we can expect full convergence rates
since we have an exact representation of the geometry.
Indeed, we realize the full convergence rate by choosing  a fine grid as the reference solution and 
comparing the solutions at each refinement step against this reference solution
as shown in Table~\ref{LMMT_sec:3:tab:torus:HeteroDeg2}.

\begin{table}[ht!]
 \centering 
 \caption{Torus: $L_{2}$ and energy norm errors with degree $k = 2.$}
  \begin{tabular}{|l|l|l|l|l|l|}
   \hline
     \cline{2-5}   
     Dofs     &$L_{2}$ error              &conv. rate  &dG error                 &conv. rate     \\  \hline
        108   &2.04841$\times 10^{6}$    &0            &8.56088$\times 10^{6}$   &0			\\
        208   &1.03579$\times 10^{6}$    &0.98377      &4.89856$\times 10^{6}$   &0.805403		\\
        504   &128215                    &3.0141       &1.09244$\times 10^{6}$   &2.1648		\\
       1480   &15431.9                   &3.05458      &258162                   &2.08121		\\
       4968   &1808.17                   &3.09331      &64084.3                  &2.01023		\\
      18088   &213.252                   &3.0839       &17361.8                  &1.88406		\\
     268840   &reference 	         &solution     &reference                &solution	\\
       \hline
   \end{tabular} 
 \label{LMMT_sec:3:tab:torus:HeteroDeg2}
\end{table}


\subsubsection{Car}
\label{LMMT_subsubsec:3.4.3:Car}


We now consider the diffusion problem on an open, free-form surface. In
order to demonstrate the fact that our results are general and not
limited to academic examples, we apply our methods to a CAD model
representing a car shell.

The surface is composed of eight quadratic B-spline patches, 
shown in Fig.~\ref{LMMT_sec:3:fig:car} (left). The model exhibits several
small bumps in the interior of the surface and sharp corners on the
boundary. In addition, the patches have varying areas and meet along
curved one-dimensional (quadratic) B-spline interfaces.

We choose a constant diffusion coefficient on the whole domain and we
prescribe homogeneous Dirichlet conditions along the boundary. For the
right-hand side, we used a (globally defined in $\mathbb R^3$) linear 
function which we restricted on the surface, in order to obtain a 
smooth solution to our problem.

As suggested by the isogeometric paradigm, we used quadratic B-spline
basis functions, forwarded on the surface, as discretization basis.
We started with a coarse grid, which was $h-$refined uniformly several
times, in both parametric directions.

An analytic formula for the exact solution is not available. 
Therefore,
we have chosen a fine grid of approximately one million degrees of
freedom as the re\-fe\-ren\-ce solution.  Comparing against this
solution (obtained by the corresponding cG IgA method), the expected
convergence rates have been observed. Table~\ref{LMMT_sec:3:tab:car}
contains the numeric results for the $L^2$ norm, obtained using either
the continuous (with strong or weak imposition of Dirichlet
boundaries) or dG IgA method. Apart from
the observed order of convergence, the magnitude of the error agrees in
all cases as well.

\begin{figure}[th!]
  \centering
\includegraphics[width=0.5\textwidth]{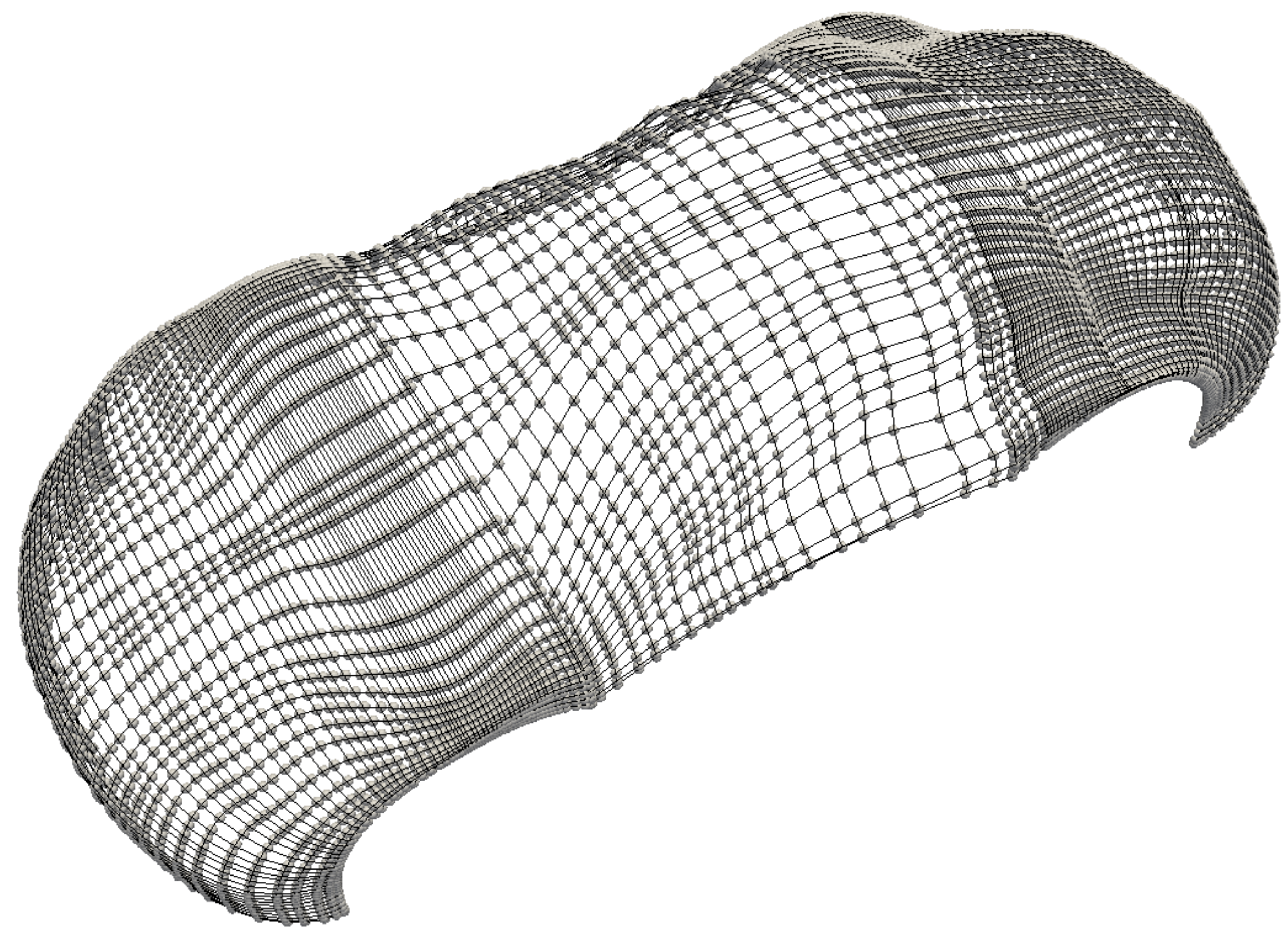}
  \includegraphics[width=0.5\textwidth]{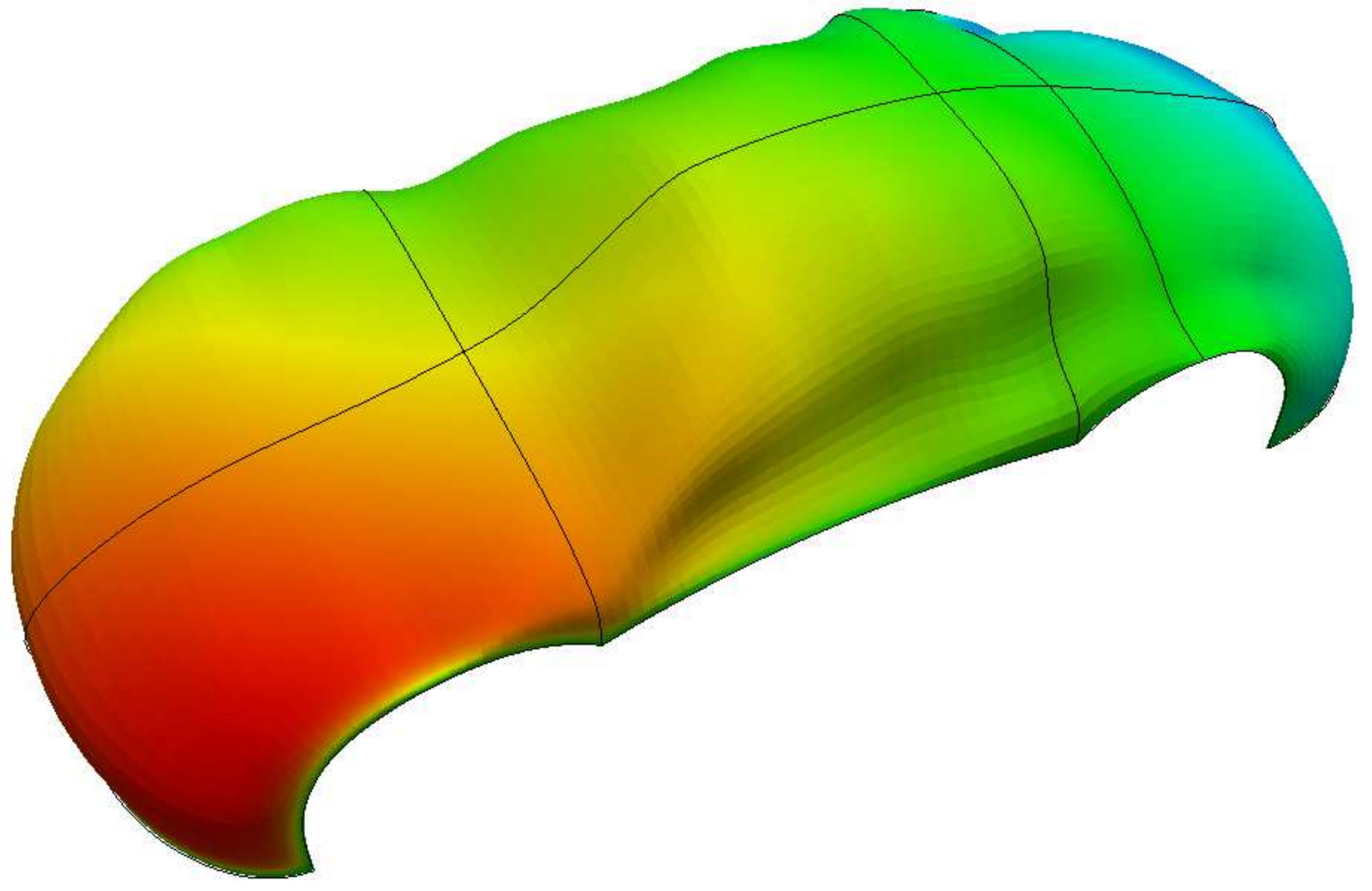}
  \caption{Car: Control points of the geometry (left),
                IgA solution of the Dirichlet surface diffusion  problem (right). Black lines indicate patch interfaces.
	  }
  \label{LMMT_sec:3:fig:car}
\end{figure}
%
\begin{table}[th!]
 \centering
  \begin{tabular}{|l||l|l||l|l||l|l|}\hline
  $k=2$ & 
  \multicolumn{2}{|c|}{cG IgA} & 
  \multicolumn{2}{|c|}{Nitsche-type BCs}&
  \multicolumn{2}{|c|}{dG IgA}
\\ \hline
 DoFs      &$L_2$ error & rate & $L_2$ error & rate & $L_2$ error & rate 	\\ \hline
         96&     1.87598&            &     1.77144&            &     1.77029&           \\
        192&     1.23006&    0.608922&     1.19201&    0.571527&     1.19188&    0.570749\\
        480&    0.633648&    0.956971&    0.623767&    0.934312&    0.623764&    0.934161\\
       1440&    0.241842&     1.38962&    0.239835&     1.37897&    0.239835&     1.37896\\
       4896&   0.0659275&     1.87511&   0.0655902&     1.87049&   0.0655902&     1.87049\\
      17952&   0.0115491&      2.5131&    0.011502&      2.5116&    0.011502&      2.5116\\
      68640&  0.00155178&     2.89578&  0.00154625&     2.89504&  0.00154624&     2.89504\\
 $\sim$ 268K &  0.00017567&   3.14298& 0.000175096&     3.14256& 0.000175094&     3.14257\\
 $\sim$ 1M   & reference & reference &  reference &  reference &  reference &     reference  \\
     \hline
\end{tabular}
\caption{Diffusion on a car-shell model. Numerical results for continuous Galerkin IgA with strong 
imposition of Dirichlet boundary (cG IgA), weak imposition (Nitsche-type) and finally 
patch-discontinuous Galerkin (dG IgA).}
\label{LMMT_sec:3:tab:car}
\end{table}
%



\section{The \gismo C++ library}
\label{LMMT_sec:5:G+SMO}

Isogeometric analysis requires seamless integration of Finite Element
Analysis (FEA) and Computer-aided design (CAD) software. The existing
software libraries, however, cannot be adapted easily to the rising
new challenges since they have been designed and developed for
different purposes. In particular, FEA codes are traditionally
implemented by means of functional programming, and are focused on
treating nodal shape function spaces. In CAD packages, on the other
hand, the central objects are free-form curves and surfaces, defined
by control points, which are realized in an object-oriented
programming environment.

\gismo is an object-oriented, template C++ library, that implements a
generic concept for IGA, based on abstract classes for geometry map,
discretization basis, assemblers, solvers and so on. It makes use of
object polymorphism and inheritance techniques in order to support a
variety of different discretization bases, namely B-spline, Bernstein,
NURBS bases, hierarchical and truncated hierarchical B-spline bases of
arbitrary polynomial order, and so on.

Our design allows the treatment of geometric entities such as surfaces
or volumes through dimension independent code, realized by means of
template meta-programming. Available features include simulations
using continuous and discontinuous Galerkin approximation of PDEs,
over conforming and non-conforming multi-patch computational
domains. PDEs on surfaces as well as integral equations arising from
elliptic boundary value problems. Boundary conditions may be imposed
both strongly and weakly. In addition to advanced discretization and
generation techniques, efficient solvers like multigrid iteration
schemes are available. Methods for solving non-linear problems are
under development. Finally, we aim to employ existing high-end
libraries for large-scale parallelization.

In the following paragraphs we shall provide more details on the
design and features of the library.

\subsection{Description of the main modules}

The library is partitioned into smaller entities, called modules. 
Currently, there are six (6) modules in \gismo namely
Core, Matrix, NURBS, Modeling, Input/Output and Solver modules.

The \textbf{Core module} is the backbone of the library. Here an abstract
interface is defined for a \emph{basis}, that is, a set of real-valued
functions living on a parameter domain. At this level, we do not
specify how these functions (or its derivatives) should be evaluated.
However, a number of virtual member functions define an interface that
should be implemented by derived classes of this type.  Another
abstract class is the \emph{geometry} class. This object consists of
a (still abstract) basis and a coefficient vector, and represents a
patch.  Note that parameter or physical dimension are not specified at
this point.  There are four classes  directly derived from the
geometry class; these are curve, surface, volume and bulk. These are
parametric objects with known parameter space dimension $1,\,2,\,3$
and $4$ respectively.

Another abstract class is a \emph{function} class. The interface for
this class includes evaluation, derivations and other related
operations. The \emph{geometry} abstract class is actually deriving
from the function class, demonstrating the fact that parametric
geometries can be simply viewed as (vector) functions. Another
interesting object is the \emph{multipatch} object.  CAD models are
composed of many patches. Therefore, a multipatch structure is of great
importance.  It contains two types of information; first, geometric
information, essentially a list of geometry patches. Second,
topological information between the patches, that is, the adjacency
graph between patch boundaries, degenerate points, and so on. Let us
also mention the \emph{field} class, which is the object that
typically represents the solution of a PDE. A \emph{field} is a
mathematical scalar or vector field which is defined on a parametric
patch, or multipatch object. It may be evaluated either on the
parameter or physical space, as the isogeometric paradigm suggests.

The \textbf{Matrix module} contains all the linear algebra related
infrastructure.  It is based on the third party library Eigen.  The
main objects are dense and sparse matrices and vectors.  
Typical matrix decompositions such as LU, QR, SVD, and so
on, are available.  Furthermore, the user has also access 
to iterative solvers like conjugate gradient methods 
with different preconditioners. Finally, one can use popular
high-end linear solver packages like PARADISO and SuperLU 
through a common interface.

The \textbf{NURBS module} consists of  
B-Splines, NURBS, B\'ezier of arbitrary degrees and knot-vectors,
tensor-product B-splines of arbitrary spatial dimension.

The \textbf{Modeling module} provides data structures and geometric
operations that are needed in order to prepare CAD data for analysis.
It contains trimmed surfaces, boundary represented (B-rep) solids and
triangle meshes. Regarding modeling 
operations, B-Spline fitting, smoothing of point clouds, 
Coon's patches, and volume segmentation methods 
are available.

The \textbf{Input/Output module} is responsible for visualization as
well as file reading and writing. For visualization, we employ Paraview
or Axel (based on VTK).  
An important issue is file formats.
Even if NURBS is an industrial standard, a variety of different
formats are used in the CAD industry to exchange NURBS data.
In \gismo, we have established I/O with popular CAD formats, which
include the 3DM file format of Rhinoceros 3D modeler, the \verb+X_T+
format of Siemens' NX platform as well as an (exported) format used by
the LS-DYNA general-purpose finite element program.

The \textbf{Assembler module} can already treat a number of PDE problems like 
convection-diffusion problems, linear elasticity, Stokes equations
as well as diffusion problems on surfaces 
by means of continuous or discontinuous Galerkin
methods, including 
divergence preserving discretizations for Stokes equations.
Strong or weak imposition Dirichlet boundary conditions 
and Neumann-type conditions are provided.
Boundary element IgA collocation techniques are also available.
%

Apart from the modules described above, there are several more which
are under development. These include a hierarchical bases module, an
optimization module and a triangular B\'ezier module.

\begin{figure}[t]
  \centering
\includegraphics[width=\textwidth]{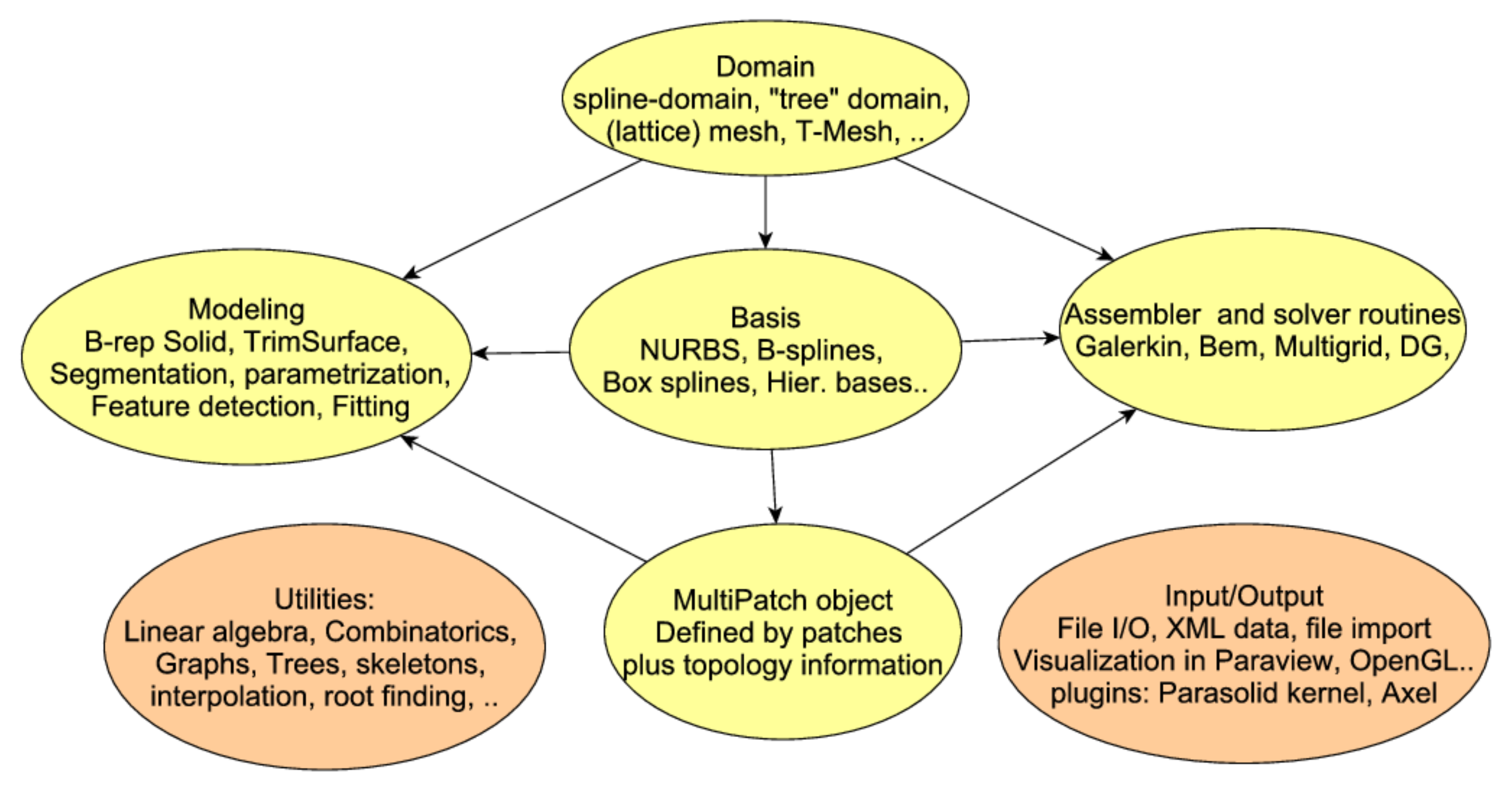}
\caption{A diagram of the different parts in the library and their interactions}
  \label{fig:gismo:general}
\end{figure}

\subsection{Development framework}

The library uses a set of standard C++ tools in order to achieve cross-platform
functionality. Some compilers that have been tested includes Microsoft's Visual C++,
Mingw32, GCC, Clang and the Intel compiler.

The main development tools used are:
    \begin{itemize}
    \item {CMake} cross-platform build system for configuration of the
      code. A small set of options are available in order to enable
      optional parts of the library that require external
      packages. Furthermore, the user is able to enable debug mode build,
      that disables code optimization and facilitates development, by
      attaching debug information in the code.

    \item The code is available in an {SVN} repository. We have chosen
      the continuous integration as development policy. That is, all
      developers commit their contributions in a mainline repository.

    \item {Trac} software management system is used for bug
      reporting. Additionally, we use the integrated wiki for
      documentation and user guidance.

    \item A {CDASH} software testing server is employed for regular
      compilation and testing of the mainline code. Nightly builds are
      executed on different platforms, as well as continuous builds
      after user commits. This allows to easily correct errors and
      ensure the quality of the library. Coverage
      analysis and memory checks are also performed regularly.

    \item The in-source documentation system Doxygen is used heavily
      in the code.

    \item A mailing list is available for communication and user support.

    \end{itemize}

    Regarding tools that are employed in the library:

   \begin{itemize}
   \item The C++ Standard Library is employed. This is available by
     default on C++ modern installations.

   \item The Eigen C++ Linear Algebra library is used for linear
     algebra operations.  This library features templated coefficient
     type, dense and sparse matrices and vectors, typical matrix
     decompositions (LU, QR, SVD, ..) as well as iterative solvers and
     wrappers for high-end packages (SuperLU, PaStiX, SparseSuite,..).

    \item {XML} reader/writer for 
      input/output of XML files, 
      and other CAD formats (OFF, STL, OBJ, GeoPDEs, 3DM, Parasolid). 

    \item Mathematical Expression Toolkit Library (ExprTk). This is an
      expression-tree evaluator for mathematical function expressions
      that allows input of functions similar to Matlab's interface.

    \end{itemize}

\subsection{Additional  features and extensions}

    Apart from the basic set of tools provided in the library, the
    user has the possibility to enable features that require
    third-party software. The following connections to external tools
    is provided:
   \begin{itemize}
    \item OpenNurbs library used to support Rhino's 3DM CAD file format.
      This allows file exchange with standard CAD software.

    \item Connection to {Parasolid} geometric kernel. Enabling this
      feature allows input and output of the \verb|x_t| file
      format. Also, the user can employ advanced modeling operations
      like intersection, trimming, boolean operations, and so on.

    \item Connection to {LS-DYNA}. The user can output simulation data
      that can be used by LS-DYNA's Generalized element module system,
      for performing for instance simulations on shells.

    \item Connection to {IPOPT} nonlinear optimization library. With
      this feature one can use powerful interior point constrained
      optimization algorithms in order to perform, for instance,
      isogeometric shape optimization.

    \item {MPFR} library for multi-precision floating point
      arithmetic. With this feature critical geometric operations can
      be performed with arbitrary precision, therefore guaranteeing a
      verified result.

    \item Graphical interface, interaction and display using {Axel modeler}.

    \end{itemize}

\subsection{Plugins for third-party platforms}

    Finally, there are currently two plugins under development. The
    plugins allow third-party software to employ and exchange data
    with \gismo.

   \begin{itemize}

   \item Axel Modeler is an open-source spline modeling package based
     on Qt and VTK. Our plugin allows to use \gismo within this
     graphical interface, and provides user interaction (e.g. control
     point editing) and display.

   \item A Matlab interface is under development. This allows to use
     operations available in \gismo  within Matlab.

    \end{itemize}


\noindent\textbf{Acknowledgement. }
This research was supported by the Research Network
``Geometry \textbf{+} Simulation'' (NFN S117), funded by
the Austrian Science Fund (FWF).

\bibliographystyle{abbrv}
\bibliography{sec_5_iga}
\end{document}